\documentclass[11pt,english]{amsart}
\usepackage[T1]{fontenc}
\usepackage[latin9]{inputenc}
\usepackage{babel}
\usepackage{verbatim}
\usepackage{amsbsy}
\usepackage{amsthm}
\usepackage{amssymb}
\usepackage{graphicx}
\usepackage[letterpaper]{geometry}
\usepackage[pdfusetitle,
 bookmarks=true,bookmarksnumbered=false,bookmarksopen=false,
 breaklinks=false,pdfborder={0 0 1},backref=false,colorlinks=false]
 {hyperref}
\hypersetup{
 colorlinks=true,linkcolor=teal,citecolor=purple,linktocpage=true}

\makeatletter
\usepackage[inline]{enumitem}
\theoremstyle{plain}
\newtheorem{thm}{\protect\theoremname}[section]
\theoremstyle{remark}
\newtheorem{rem}[thm]{\protect\remarkname}
\theoremstyle{plain}
\newtheorem{cor}[thm]{\protect\corollaryname}
\theoremstyle{definition}
\newtheorem{defn}[thm]{\protect\definitionname}
\theoremstyle{plain}
\newtheorem{question}[thm]{\protect\questionname}
\theoremstyle{plain}
\newtheorem{lem}[thm]{\protect\lemmaname}
\theoremstyle{definition}
\newtheorem{example}[thm]{\protect\examplename}
\theoremstyle{plain}
\newtheorem{prop}[thm]{\protect\propositionname}

\@ifundefined{date}{}{\date{}}
\usepackage{amsmath, tikz-cd, amssymb, placeins}
\usepackage{fnpct}
\usepackage{listings}
\numberwithin{figure}{section}
\numberwithin{equation}{section}
\DeclareRobustCommand*\cal{\@fontswitch\relax\mathcal}

\usetikzlibrary{calc}
\usetikzlibrary{decorations.pathmorphing}

\tikzset{curve/.style={settings={#1},to path={(\tikztostart)
    .. controls ($(\tikztostart)!\pv{pos}!(\tikztotarget)!\pv{height}!270:(\tikztotarget)$)
    and ($(\tikztostart)!1-\pv{pos}!(\tikztotarget)!\pv{height}!270:(\tikztotarget)$)
    .. (\tikztotarget)\tikztonodes}},
    settings/.code={\tikzset{quiver/.cd,#1}
        \def\pv##1{\pgfkeysvalueof{/tikz/quiver/##1}}},
    quiver/.cd,pos/.initial=0.35,height/.initial=0}

\tikzset{tail reversed/.code={\pgfsetarrowsstart{tikzcd to}}}
\tikzset{2tail/.code={\pgfsetarrowsstart{Implies[reversed]}}}
\tikzset{2tail reversed/.code={\pgfsetarrowsstart{Implies}}}
\tikzset{no body/.style={/tikz/dash pattern=on 0 off 1mm}}

\makeatother

\providecommand{\corollaryname}{Corollary}
\providecommand{\definitionname}{Definition}
\providecommand{\examplename}{Example}
\providecommand{\lemmaname}{Lemma}
\providecommand{\propositionname}{Proposition}
\providecommand{\questionname}{Question}
\providecommand{\remarkname}{Remark}
\providecommand{\theoremname}{Theorem}

\begin{document}
\title{Basepoints in Khovanov homology and nonorientable surfaces}
\author{Gheehyun Nahm}
\thanks{The author was partially supported by the ILJU Academy and Culture
Foundation, the Simons collaboration \emph{New structures in low-dimensional
topology}, and a Princeton Centennial Fellowship.}
\address{Department of Mathematics, Princeton University, Princeton, New Jersey
08544, USA}
\email{gn4470@math.princeton.edu}
\begin{abstract}
We enhance the Khovanov TQFT using basepoint actions, over the field
with two elements. Our enhanced Khovanov TQFT behaves similarly to
gauge/Floer theoretic invariants of the double branched cover with
opposite orientation: they both are invariant, in a certain sense,
under taking the connected sum with the standard $\mathbb{RP}^{2}$
with Euler number $-2$, and they both vanish after taking the connected
sum with the standard $\mathbb{RP}^{2}$ with Euler number $2$. This
invariance property answers a version of a question posed by Lipshitz
and Sarkar. Furthermore, our construction establishes, as a special
case, functoriality for the pointed Khovanov homology defined by Baldwin,
Levine, and Sarkar.
\end{abstract}

\maketitle
\tableofcontents{}

\section{\label{sec:Introduction}Introduction}

The Khovanov TQFT \cite{MR1740682,Ja,MWW} is a powerful combinatorial
invariant that assigns to a link $L$ in $S^{3}$ a bigraded group
$Kh(L)$, and assigns to a link cobordism $\Sigma$ in $I\times S^{3}$
a homomorphism $Kh(\Sigma):Kh(L)\to Kh(L')$ that is invariant under
isotoping $\Sigma$ rel.\ $\partial$. Throughout this paper, we
work over the field $\mathbb{F}:=\mathbb{Z}/2\mathbb{Z}$; thus $Kh(L)$
is a bigraded $\mathbb{F}$-vector space, and in particular a surface
$\Sigma$ in $D^{4}$ induces a map $Kh(\Sigma):\mathbb{F}\to Kh(\partial\Sigma)$.

The Khovanov TQFT has been successfully used to study questions about
\emph{orientable} surfaces in certain $4$-manifolds \cite{Ra,swannthesis,MR3189434,MR4076631,MR4541332,MR4562563,MR4726569,hayden2023seifertsurfaces4ball,guth2023doubleddiskssatellitesurfaces,ren2024khovanovv3,teng2025endkhovanovhomologyexotic}.
In contrast, applications of the Khovanov TQFT to \emph{nonorientable}
surfaces are comparatively sparse; notable exceptions include \cite{ballinger2020concordance,MR4504654,hayden2023seifertsurfaces4ball}.
A major difficulty is that the Khovanov link cobordism maps often
vanish for nonorientable surfaces.

In this paper, we construct, using \emph{basepoint actions}, an enhancement
of the Khovanov TQFT (Theorem~\ref{thm:enhanced-khovanov}) that
circumvents this vanishing phenomenon (Theorem~\ref{thm:main-thm}).
Note that our enhancement is also purely combinatorial.

\begin{thm}[Corollary~\ref{cor:Let--be}]
\label{thm:main-thm}Let $K$ be a knot in $S^{3}$ and let $\Sigma_{1},\Sigma_{2}$
be two properly embedded orientable surfaces in $D^{4}$ with boundary
$K$. If $Kh(\Sigma_{1})\neq Kh(\Sigma_{2})$, then for any $N\ge0$,
$\Sigma_{1}\#N\mathbb{RP}^{2}$ and $\Sigma_{2}\#N\mathbb{RP}^{2}$
are distinguished by the enhanced Khovanov link cobordism map, where
$\mathbb{RP}^{2}$ denotes the standard $\mathbb{RP}^{2}$ in $S^{4}$
with Euler number $-2$.
\end{thm}

\begin{rem}
For all $N\ge1$, the original Khovanov link cobordism map vanishes
for $\Sigma_{i}\#N\mathbb{RP}^{2}$.
\end{rem}

Hayden and Sundberg \cite{MR4726569} proved that the Khovanov link
cobordism maps distinguish certain exotic orientable surfaces in $D^{4}$.
Theorem~\ref{thm:main-thm}, combined with their results, yields
Corollary~\ref{cor:examples}, which we prove in Section~\ref{sec:Standard-'s-with}.
\begin{cor}
\label{cor:examples}All of Hayden and Sundberg's exotic pairs of
surfaces \cite[Theorem 1.1]{MR4726569} remain nonisotopic rel.~$\partial$
after taking the connected sum with arbitrarily many copies of the
standard $\mathbb{RP}^{2}$ with Euler number~$2$.\footnote{There is a difference in Euler number between Theorem~\ref{thm:main-thm}
and Corollary~\ref{cor:examples} because our convention for Khovanov
homology differs from that of Hayden and Sundberg \cite{MR4726569};
see the proof of Corollary~\ref{cor:examples} in Section~\ref{sec:Standard-'s-with}
(compare \cite[Conventions]{nahm2025khovanov}).}
\end{cor}

Theorem~\ref{thm:main-thm} also answers a version of a question
that Lipshitz and Sarkar posed in the paper where they gave the first
gauge theory-free proof that pairs of exotic nonorientable surfaces
exist \cite[Theorem~1.2 and Question~6]{MR4504654}. Their strategy
was to start with a specific exotic pair of slice disks $\Sigma_{1},\Sigma_{2}$
such that $Kh(\Sigma_{1})\neq Kh(\Sigma_{2})$ (proved by Hayden and
Sundberg) and to append a nonorientable link cobordism $C:K\to K'$
such that $Kh(C)\circ Kh(\Sigma_{1})\neq Kh(C)\circ Kh(\Sigma_{2})$.
This led them to ask whether this can be done in general, i.e.\ whether
such nonorientable $C$ exists for any pair of disks $\Sigma_{1},\Sigma_{2}$
such that $Kh(\Sigma_{1})\neq Kh(\Sigma_{2})$, and whether $C$ can
be chosen to have crosscap number at least $3$. By Theorem~\ref{thm:main-thm},
for our enhanced TQFT, we can take $C$ to be the connected sum of
the identity cobordism with $N$ copies of the standard $\mathbb{RP}^{2}$
with Euler number $-2$ for any $N\ge1$.

The rest of the introduction is organized as follows. In Subsections~\ref{subsec:An-enhancement-of}~and~\ref{subsec:An-example},
we describe a special case of our enhanced Khovanov TQFT, survey related
results, and give an example where it sees more information than the
original Khovanov TQFT. In Subsection~\ref{subsec:Motivation}, we
compare our enhancement to the original Khovanov TQFT and gauge theoretic
invariants of the double branched cover.

\subsection{\label{subsec:An-enhancement-of}An enhancement of the Khovanov TQFT}

Basepoints on link diagrams give rise to extra structure on the Khovanov
chain complex. Fix $n\ge0$.
\begin{defn}[Generic $n$-pointed links]
View $S^{3}$ as $\mathbb{R}^{3}\cup\{\infty\}$. A link $L\subset S^{3}$
is \emph{generic} if $L$ avoids $\infty\in S^{3}$ and if the projection
$\mathbb{R}^{3}\to\mathbb{R}^{2}$, $(x,y,z)\mapsto(x,y)$, gives
rise to a link diagram of $L$. A \emph{generic $n$-pointed link
}$(L,\vec{p})$ is a generic link $L$ together with an $n$-tuple
$\vec{p}=(p_{1},\cdots,p_{n})$ of points on $L$, away from the crossings.
These points $p_{i}$ are called \emph{basepoints}.
\end{defn}

If $(L,\vec{p})$ is a generic $n$-pointed link, then $\vec{p}$
induces \cite{MR2034399,MR3190305} an $R_{n}$-module structure on
$CKh(L)$, where $R_{n}:=\mathbb{F}[X_{1},\cdots,X_{n}]/(X_{1}^{2},\cdots,X_{n}^{2})$.
This has been studied in various contexts by, for instance, Hedden-Ni
\cite{MR3190305}, Baldwin-Levine-Sarkar \cite{MR3604486}, and Lipshitz-Sarkar
\cite{MR4521052}. Hedden-Ni considered $Kh(L)$ as an $R_{n}$-module,
Baldwin-Levine-Sarkar considered the chain complex $CKh(L)\otimes_{R_{n}}\bigotimes_{i=1}^{n}(R_{n}\xrightarrow{X_{i}}R_{n})$
whose homology they call \emph{pointed Khovanov homology}, and Lipshitz-Sarkar
worked in full generality, by viewing $Kh(L)$ as an $A_{\infty}$-module
over $R_{n}$. (See Appendix~\ref{sec:Comparison-with-previous}
for further discussion.) 

Moreover, each of them showed that the isomorphism type of their invariant
depends only on the isotopy class of the pointed link $(L,\vec{p})$.
To prove this invariance, they defined maps for a minimally required
class of \emph{decorated link cobordisms} $(\Sigma,\vec{A}):(L,\vec{p})\to(L',\vec{p'})$,
which are pairs of a link cobordism $\Sigma:L\to L'$ in $I\times S^{3}$
and an $n$-tuple $\vec{A}=(A_{1},\cdots,A_{n})$ of embedded arcs
$A_{i}\subset\Sigma$ such that $\partial A_{i}=\{p_{i},p_{i}'\}$.
The decorated link cobordisms that they consider are those arising
from performing a Reidemeister move away from the basepoints or sliding
a basepoint across a crossing (we call these \emph{Reidemeister movies}
and \emph{slide movies}, respectively; see Definition~\ref{def:elementary-movies}).

Our contribution is establishing functoriality for these pointed link
invariants (Theorem~\ref{thm:enhanced-khovanov}). Specifically,
we define maps for arbitrary decorated link cobordisms that recover
the existing maps for Reidemeister and slide movies, and prove that
our cobordism maps depend only on the isotopy rel.\ $\partial$ class
of the cobordism.

To make this precise, let $\mathsf{Link}$ be the category of generic
links $L\subset S^{3}$ and isotopy rel.\ $\partial$ classes of
link cobordisms in $I\times S^{3}$. Then, the Khovanov chain-level
TQFT is a functor $CKh:\mathsf{Link}\to K^{b}(\mathsf{Mod}_{\mathbb{F}})$
to the bounded homotopy category of chain complexes of $\mathbb{F}$-modules.
The enhanced TQFT is $DKh:n\mathsf{Link}\to D^{b}(\mathsf{Mod}_{R_{n}})$
where $n\mathsf{Link}$ is the category of generic $n$-pointed links
and isotopy rel.\ $\partial$ classes of decorated link cobordisms,
and $D^{b}(\mathsf{Mod}_{R_{n}})$ is the bounded derived category
of $R_{n}$-modules.

Theorem~\ref{thm:enhanced-khovanov} is our main functoriality statement;
we prove it at the end of Section~\ref{sec:Invariance}. In Section~\ref{sec:remarks}
we discuss the necessity of the derived category and explain an observation
that led us to Theorems~\ref{thm:main-thm}~and~\ref{thm:enhanced-khovanov}.
\begin{thm}[The enhanced Khovanov TQFT]
\label{thm:enhanced-khovanov}There is a functor $DKh:n\mathsf{Link}\to D^{b}(\mathsf{Mod}_{R_{n}})$
such that the following hold.
\begin{enumerate}
\item \label{enu:141}On the object level, $DKh(L,\vec{p})$ is $CKh(L)$
equipped with the $R_{n}$-action induced by $\vec{p}$.
\item \label{enu:144}The decorated link cobordism map $DKh(\Sigma,\vec{A})$
recovers the original map $CKh(\Sigma)$ via the forgetful functor
$D^{b}(\mathsf{Mod}_{R_{n}})\to D^{b}(\mathsf{Mod}_{\mathbb{F}})\cong K^{b}(\mathsf{Mod}_{\mathbb{F}})$.
\item \label{enu:142}$DKh(\Sigma,\vec{A})$ only depends on the isotopy
rel.\ $\partial$ class of $\Sigma$ and the mod $2$ homology classes
$[A_{i},\partial A_{i}]\in H_{1}(\Sigma,\{p_{i},p_{i}'\};\mathbb{F})$.
\item \label{enu:For-the-decorated}If $(\Sigma,\vec{A})$ is a Reidemeister
or slide movie, then $DKh(\Sigma,\vec{A})$ recovers the maps defined
by Hedden-Ni \cite{MR3190305}, Baldwin-Levine-Sarkar \cite{MR3604486},
and Lipshitz-Sarkar \cite{MR4521052}.
\end{enumerate}
\end{thm}

\begin{rem}
In Section~\ref{sec:Pointed-links-and} we construct $DKh$ for a
more general class of decorated link cobordisms $(\Sigma,\vec{A})$,
where, in particular, we allow $A_{i}$ to be any immersed $1$-manifold.
Later, we specialize this general construction and define a functor
(Definition~\ref{def:hkh}), which we call $HKh$, from the category
of generic links $L$ and isotopy rel.\ $\partial$ classes of link
cobordisms $\Sigma:L\to L'$ decorated by a class $w\in H_{1}(\Sigma;\mathbb{F})$.
On objects we have
\[
HKh(L)=Kh(L)\otimes_{\mathbb{F}}\mathbb{F}[\xi,\xi^{-1}]/\xi\mathbb{F}[\xi]
\]
where $\xi$ is a formal variable with bidegree $(1,2)$. If $\Sigma$
is a link cobordism, then 
\[
HKh(\Sigma\#N\mathbb{RP}^{2},w)=Kh(\Sigma)\otimes_{\mathbb{F}}\xi^{N}
\]
where $w$ is the sum of the homology classes of the core $\mathbb{RP}^{1}$s
of the $N$ $\mathbb{RP}^{2}$ summands, and $\xi^{N}:\mathbb{F}[\xi,\xi^{-1}]/\xi\mathbb{F}[\xi]\to\mathbb{F}[\xi,\xi^{-1}]/\xi\mathbb{F}[\xi]$
denotes multiplication by $\xi^{N}$.
\end{rem}

\subsection{\label{subsec:An-example}The standard $\mathbb{RP}^{2}$ with Euler
number $-2$}

\begin{figure}[h]
\begin{centering}
\includegraphics[scale=2.2]{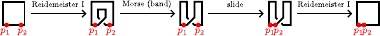}
\par\end{centering}
\caption{\label{fig:rp2-schematic}A movie for a decorated cobordism that corresponds
to the standard $\mathbb{RP}^{2}$ with Euler number $-2$.}
\end{figure}

Let us give a simple, representative example of how $DKh(\Sigma,\vec{A})$
contains strictly more information than $CKh(\Sigma)$. Consider the
decorated cobordism $(\Sigma,(A_{1},A_{2})):(U,(p_{1},p_{2}))\to(U,(p_{1},p_{2}))$
of Figure~\ref{fig:rp2-schematic}, where $U\subset S^{3}$ is the
unknot. Note that $\Sigma$ is obtained from the standard $\mathbb{RP}^{2}$
with Euler number $-2$ by removing two disks. The Khovanov chain
complex $CKh(U)$ is homotopy equivalent to $\mathbb{F}\oplus\mathbb{F}x$
with trivial differential, and $CKh(\Sigma):\mathbb{F}\oplus\mathbb{F}x\to\mathbb{F}\oplus\mathbb{F}x$
is identically zero.

However, it turns out that $DKh(\Sigma,(A_{1},A_{2}))\in\mathrm{Hom}_{D^{b}(\mathsf{Mod}_{R_{2}})}(CKh(U),CKh(U))$
is nonzero. Since this map lives in the bounded derived category $D^{b}(\mathsf{Mod}_{R_{2}})$
of $R_{2}=\mathbb{F}[X_{1},X_{2}]/(X_{1}^{2},X_{2}^{2})$-modules,
it may not come from a chain map $CKh(U)\to CKh(U)$ (indeed, it does
not), but if $F^{\bullet}$ is a free $R_{2}$-resolution of $CKh(U)$,
then $DKh(\Sigma,(A_{1},A_{2}))$ comes from a chain map $F^{\bullet}\to F^{\bullet}$.

Hence, let us find such an $F^{\bullet}$. As $R_{2}$-chain complexes,
$CKh(U)$ is $R_{2}/(X_{1}+X_{2})$ with trivial differential, and
so the following is a free $R_{2}$-resolution of $CKh(U)$: 
\[
\cdots\xrightarrow{X_{1}+X_{2}}h^{-2}q^{-4}R_{2}\xrightarrow{X_{1}+X_{2}}h^{-1}q^{-2}R_{2}\xrightarrow{X_{1}+X_{2}}R_{2}.
\]
It turns out that the map $DKh(\Sigma,(A_{1},A_{2}))$ is the following
dashed map (our convention is that chain maps need not preserve the
homological grading; see Subsection~\ref{subsec:Conventions}):
\[\begin{tikzcd}[ampersand replacement=\&]
	{\cdots } \& {h^{-2}q^{-4}R_2 } \& {h^{-1}q^{-2}R_2 } \& {R_2} \& {\simeq CKh(U)} \\
	{\cdots } \& {h^{-2}q^{-4}R_2 } \& {h^{-1}q^{-2}R_2 } \& {R_2 } \& {\simeq CKh(U)}
	\arrow["{X_1 +X_2 }", from=1-1, to=1-2]
	\arrow["1"{description}, dashed, from=1-1, to=2-2]
	\arrow["{X_1 +X_2 }", from=1-2, to=1-3]
	\arrow["1"{description}, dashed, from=1-2, to=2-3]
	\arrow["{X_1 +X_2 }", from=1-3, to=1-4]
	\arrow["1"{description}, dashed, from=1-3, to=2-4]
	\arrow["{DKh(\Sigma,(A_{1},A_{2}))}", dashed, from=1-5, to=2-5]
	\arrow["{X_1 +X_2 }", from=2-1, to=2-2]
	\arrow["{X_1 +X_2 }", from=2-2, to=2-3]
	\arrow["{X_1 +X_2 }", from=2-3, to=2-4]
\end{tikzcd}\]

This map is nonzero in $\mathrm{Hom}_{D^{b}(\mathsf{Mod}_{R_{2}})}(R_{2}/(X_{1}+X_{2}),R_{2}/(X_{1}+X_{2}))$.
One way to see this is by tensoring with $\mathbb{F}[X]/X^{2}$ where
$X_{1},X_{2}\in R_{2}$ act as $X$, and then taking homology; this
gives the following.
\[\begin{tikzcd}[ampersand replacement=\&,column sep=tiny]
	{\cdots } \& {h^{-2}q^{-4} \mathbb F[X]/X^2 } \& {h^{-1}q^{-2}\mathbb F[X]/X^2 } \& {\mathbb F[X]/X^2 } \& {\cong H( CKh(U) {\otimes}_{R_{2}}^L\mathbb{F}[X]/X^{2})} \\
	{\cdots } \& {h^{-2}q^{-4} \mathbb F[X]/X^2 } \& {h^{-1}q^{-2}\mathbb F[X]/X^2 } \& {\mathbb F[X]/X^2 } \& {\cong H( CKh(U) {\otimes}_{R_{2}}^L \mathbb{F}[X]/X^{2})}
	\arrow["1"{description}, dashed, from=1-2, to=2-3]
	\arrow["1"{description}, dashed, from=1-3, to=2-4]
	\arrow["{H(DKh(\Sigma,(A_{1},A_{2})) {\otimes}_{R_{2}}^L \mathbb{F}[X]/X^{2})}", dashed, from=1-5, to=2-5]
\end{tikzcd}\]
\begin{rem}[Standard $\mathbb{RP}^{2}$ with Euler number $2$]
If instead $\Sigma$ came from the unknotted $\mathbb{RP}^{2}$ with
Euler number $2$, i.e.\ if the crossings of Figure~\ref{fig:rp2-schematic}
were the opposite crossings, then the induced map $DKh(\Sigma,(A_{1},A_{2}))$
is zero.
\end{rem}

\subsection{\label{subsec:Motivation}Motivation}

Khovanov homology is closely related to gauge and Floer theoretic
invariants of the orientation reverse of the double branched cover
of $S^{3}$ along $L$, as highlighted by the various spectral sequences
\cite{MR2141852,MR2764887,MR3394316,daemi2015abeliangaugetheoryknots}
that relate these. In gauge theory and Floer theory, there exist interesting
invariants for \emph{closed} $4$-manifolds, such as the Donaldson
and Seiberg-Witten invariants \cite{MR892034,DONALDSON1990257,MR1293681,MR1306869,MR1306021}
as well as the Ozsv\'{a}th-Szab\'{o} \emph{mixed invariant} \cite{MR2222356,MR4337438},
which is a Heegaard Floer analogue of the Seiberg-Witten invariant
(compare \cite{Kronheimer_Mrowka_2007}). Inspired by this mixed invariant,
Lipshitz and Sarkar \cite{MR4504654} defined \emph{mixed invariants}
in Khovanov homology for certain nonorientable link cobordisms.

One would hope \cite[Question~7]{MR4504654} that the Khovanov mixed
invariant behaves similarly to the Seiberg-Witten invariant and the
mixed invariant of the orientation reverse of the double branched
cover. However, the Khovanov link cobordism maps behave fundamentally
differently from gauge and Floer theoretic cobordism maps of the orientation
reverse of the double branched cover, even in the simplest case. Gauge
and Floer theoretic invariants are invariant, in a certain sense,
under taking the connected sum with $\overline{\mathbb{CP}^{2}}$,
and they vanish after taking the connected sum with $\mathbb{CP}^{2}$.
Taking the connected sum of a $4$-manifold with $\overline{\mathbb{CP}^{2}}$
(resp.\ $\mathbb{CP}^{2}$) corresponds to taking the connected sum
of a link cobordism $\Sigma$ with the standard $\mathbb{RP}^{2}$
with normal Euler number $-2$ (resp.\ $2$).\footnote{Let $X$ be the orientation reverse of the double branched cover of
$I\times S^{3}$ along $\Sigma$. Then, the orientation reverse of
the double branched cover of $I\times S^{3}$ along the connected
sum $\Sigma\#\mathbb{RP}^{2}$ of $\Sigma$ and the standard $\mathbb{RP}^{2}$
with Euler number $-2$ (resp.\ $2$) is $X\#\overline{\mathbb{CP}^{2}}$
(resp.\ $X\#\mathbb{CP}^{2}$).} However, the Khovanov link cobordism map, and in fact consequently
the Khovanov mixed invariant, for $\Sigma\#\mathbb{RP}^{2}$ is identically
zero for both Euler numbers $\pm2$.

Our enhanced Khovanov TQFT behaves similarly to gauge and Floer theoretic
invariants of the orientation reverse of the double branched cover:
both are invariant, in a certain sense, under taking the connected
sum with the standard $\mathbb{RP}^{2}$ with Euler number $-2$,
and they vanish after taking the connected sum with the standard $\mathbb{RP}^{2}$
with Euler number $2$.

We hope that our enhancement of the Khovanov TQFT can be used to define
a Khovanov mixed invariant that behaves more similarly to gauge and
Floer theoretic invariants; Question~\ref{que:mixed-inv} is a motivating
open question. Note that gauge theoretic invariants of the double
branched cover have been used to distinguish certain exotic closed
surfaces in $S^{4}$ \cite{MR903734,MR970078}.
\begin{question}
\label{que:mixed-inv}Is it possible to define a mixed invariant from
the Khovanov TQFT that distinguishes exotic closed surfaces in $S^{4}$?
\end{question}

\subsection{Organization}

In Section~\ref{sec:Preliminaries} we review Bar-Natan's tangle
invariant \cite{BN1} and the basepoint sliding homotopy. In Section~\ref{sec:Pointed-links-and},
we define the enhanced Khovanov TQFT, and we show that it is well-defined
in Section~\ref{sec:Invariance}. In Section~\ref{sec:Standard-'s-with},
we carry out a model computation for the standard $\mathbb{RP}^{2}$
with Euler number $-2$ and prove Theorem~\ref{thm:main-thm} and
Corollary~\ref{cor:examples}. In Section~\ref{sec:remarks} we
discuss the necessity of working in the derived category and explain
an observation that led us to our main theorems. In Appendix~\ref{sec:collapsing-colors-proof}
we prove an algebraic statement that we use in Section~\ref{sec:Invariance}.
In Appendix~\ref{sec:Comparison-with-previous} we prove Theorem~\ref{thm:enhanced-khovanov}~(\ref{enu:For-the-decorated}).

\subsection{\label{subsec:Conventions}Conventions}

The Khovanov chain complex is $(h,q)$-bigraded. The differential
has bigrading $(1,0)$, and the basepoint actions have bigrading $(0,-2)$.
In this context, we say homology (resp.\ chain complex) to mean cohomology
(resp.\ cochain complex).

If $M$ is an $R$-module and $r\in R$, then we denote multiplication
by $r$ also as $r:M\to M$.

All chain complexes $C^{\bullet}$ are bounded above, i.e.\ $C^{n}=0$
for $n\gg0$.

If $C,D$ are $R$-chain complexes and $f,g:C\to D$ are $R$-chain
maps, then $f\sim g$ means that they are $R$-chain homotopic.

Chain maps need not be homogeneous with respect to the homological
grading. Hence, the morphisms of e.g.\ $K^{b}(\mathsf{Mod}_{\mathbb{F}})$
and $D^{b}(\mathsf{Mod}_{R})$ need not be homogeneous. In other words,
e.g.\ if $\mathrm{Hom}_{D^{b}(\mathsf{Mod}_{R})}^{0}(C,D)$ denotes
the $R$-module of morphisms that preserve the homological grading,
then $\mathrm{Hom}_{D^{b}(\mathsf{Mod}_{R})}(C,D)=\bigoplus_{n\in\mathbb{Z}}\mathrm{Hom}_{D^{b}(\mathsf{Mod}_{R})}^{0}(C,h^{n}D)$
where $h^{n}$ denotes a shift in homological grading by $n$.

\subsection*{Acknowledgements}

We thank Robert Lipshitz and Peter Ozsv\'{a}th for their continuous
support and helpful discussions. We thank William Ballinger, Ines
Borchers Arias, Kunal Chawla, Glen Lim, Yuta Nakayama, Jacob Rasmussen,
Qiuyu Ren, Sucharit Sarkar, Evan Scott, Joshua Wang, and Hongjian
Yang for helpful discussions. We thank Robert Lipshitz and Sucharit
Sarkar for discussions about the sweep-around move. The author used
ChatGPT and Gemini for wording suggestions and proofreading.

\section{\label{sec:Preliminaries}Preliminaries}

In this section, we set up notations for Bar-Natan's tangle invariant
\cite{BN1} and review the basepoint sliding homotopy. For simplicity,
we only consider oriented tangles in $D^{3}=D^{2}\times[-1,1]$. Tangle
cobordisms are in $I\times D^{3}$; we call the $I$ direction the
\emph{time} direction. Let $\boldsymbol{e}\subset D^{2}\times\{0\}\subset D^{2}\times[-1,1]=D^{3}$
be an oriented $0$-manifold.

\subsection{\label{subsec:Bar-Natan's-tangle-invariant}Bar-Natan's tangle invariant}
\begin{defn}[The category $\mathsf{Tang}_{\boldsymbol{e}}$]
An oriented tangle $T\subset D^{3}$ with endpoints $\boldsymbol{e}$
is \emph{generic} if the projection $D^{2}\times[-1,1]\to D^{2}$
gives rise to a tangle diagram of $T$. The category $\mathsf{Tang}_{\boldsymbol{e}}$
has objects generic oriented tangles with endpoints $\boldsymbol{e}$
and morphisms isotopy rel.\ $\partial$ classes of tangle cobordisms
$\Sigma:T\to T'$ in $I\times D^{3}$. Tangle cobordisms need not
be orientable, and need not be compatible with the orientations of
$T,T'$.
\end{defn}

Bar-Natan's tangle invariant is a functor $C_{\mathsf{BN}}:\mathsf{Tang}_{\boldsymbol{e}}\to K^{b}(\mathsf{BN}_{\boldsymbol{e}})$;
let us explain our conventions for $\mathsf{BN}_{\boldsymbol{e}}$.
Let $\mathcal{C}ob_{\bullet/l}^{3}(\boldsymbol{e})$ be Bar-Natan's
preadditive category \cite{BN1} of dotted planar cobordisms modulo
certain relations. The $\mathrm{Hom}$ groups of $\mathcal{C}ob_{\bullet/l}^{3}(\boldsymbol{e})$
are quantum $\mathbb{Z}$-graded and the quantum grading-shifted objects
are formally added to $\mathcal{C}ob_{\bullet/l}^{3}(\boldsymbol{e})$.
Let $\mathcal{C}ob_{\bullet/l}^{3}(\boldsymbol{e})\otimes_{\mathbb{Z}}\mathbb{F}$
be the quantum $\mathbb{Z}$-graded $\mathbb{F}$-linear category
whose objects are the same as $\mathcal{C}ob_{\bullet/l}^{3}(\boldsymbol{e})$
and 
\[
\mathrm{Hom}_{\mathcal{C}ob_{\bullet/l}^{3}(\boldsymbol{e})\otimes_{\mathbb{Z}}\mathbb{F}}(X,Y):=\mathrm{Hom}_{\mathcal{C}ob_{\bullet/l}^{3}(\boldsymbol{e})}(X,Y)\otimes_{\mathbb{Z}}\mathbb{F}.
\]

\begin{defn}
The\emph{ Bar-Natan category} $\mathsf{BN}_{\boldsymbol{e}}$ is the\emph{
additive enlargement }of $\mathcal{C}ob_{\bullet/l}^{3}(\boldsymbol{e})\otimes_{\mathbb{Z}}\mathbb{F}$.\footnote{Bar-Natan denotes the additive enlargement of a preadditive category
$\mathsf{A}$ as $\mathrm{Mat}(\mathsf{A})$.}
\end{defn}

Our convention for the additive enlargement is closer to \cite[Section~3k]{MR2441780}
than to \cite{BN1}: if $\mathsf{A}$ is a quantum $\mathbb{Z}$-graded
$\mathbb{F}$-linear category, then the objects of its additive enlargement
$\Sigma\mathsf{A}$ are triples $X=(I,\{X_{i}\},\{V_{i}\})$ written
as $X=\bigoplus_{i\in I}X_{i}\otimes_{\mathbb{F}}V_{i}$ where $I$
is a finite set, $\{X_{i}\}_{i\in I}$ is a family of objects of $\mathsf{A}$,
and $\{V_{i}\}_{i\in I}$ is a family of quantum $\mathbb{Z}$-graded
vector spaces. The morphism spaces are 
\[
\mathrm{Hom}_{\Sigma\mathsf{A}}\left(\bigoplus_{i\in I}X_{i}\otimes_{\mathbb{F}}V_{i},\bigoplus_{j\in J}Y_{j}\otimes_{\mathbb{F}}W_{j}\right):=\bigoplus_{i\in I,j\in J}\mathrm{Hom}_{\mathsf{A}}(X_{i},Y_{j})\otimes_{\mathbb{F}}\mathrm{Hom}_{\mathbb{F}}(V_{i},W_{j}).
\]
Note that the morphism spaces are quantum $\mathbb{Z}$-graded. Composition
is defined by combining the composition of maps in $\mathsf{A}$ and
maps between vector spaces.

Recall \cite[Section~5]{BN1} that an oriented planar arc diagram
$D$ with input endpoints $\boldsymbol{e}_{1},\cdots,\boldsymbol{e}_{d}$
and output endpoints $\boldsymbol{e}$ induces functors 
\[
\mathsf{Tang}_{\boldsymbol{e}_{1}}\times\cdots\times\mathsf{Tang}_{\boldsymbol{e}_{d}}\to\mathsf{Tang}_{\boldsymbol{e}}\ \mathrm{and}\ K^{b}(\mathsf{BN}_{\boldsymbol{e}_{1}})\times\cdots\times K^{b}(\mathsf{BN}_{\boldsymbol{e}_{d}})\to K^{b}(\mathsf{BN}_{\boldsymbol{e}}).
\]
These functors are associative and commute with $C_{\mathsf{BN}}:\mathsf{Tang}_{\boldsymbol{f}}\to K^{b}(\mathsf{BN}_{\boldsymbol{f}})$.
For simplicity of notation, for chain complexes $C_{i}\in K^{b}(\mathsf{BN}_{\boldsymbol{e}_{i}})$,
we denote the image of $(C_{1},\cdots,C_{d})$ under the second functor
as $C_{1}\otimes\cdots\otimes C_{d}$, suppressing $D$ from the notation.

\subsection{\label{subsec:The-basepoint-action}The basepoint action and the
basepoint sliding homotopy}

Let $T\in\mathsf{Tang}_{\boldsymbol{e}}$, let $C$ be the set of
crossings of $T$, and for $\varepsilon\in\{0,1\}^{C}$ let $T_{\varepsilon}\in\mathsf{BN}_{\boldsymbol{e}}$
be the $\varepsilon$-resolution. Recall that the chain complex $C_{\mathsf{BN}}(T)\in K^{b}(\mathsf{BN}_{\boldsymbol{e}})$,
ignoring the differential, is the direct sum of the $T_{\varepsilon}$'s
(with various bigrading shifts).
\begin{defn}[Basepoint action and basepoint sliding homotopy \cite{MR2034399,MR3190305,MR3604486,MR4521052}]
Let $p\in T$ be a point away from the crossings. For each $\varepsilon$-resolution
$T_{\varepsilon}\in\mathsf{BN}_{\boldsymbol{e}}$ of $T$, let $f_{\varepsilon}\in\mathrm{Hom}_{\mathsf{BN}_{\boldsymbol{e}}}(T_{\varepsilon},T_{\varepsilon})$
be the identity cobordism $\Sigma_{\varepsilon}:=[0,1]\times T_{\varepsilon}$
together with a dot at $(0.5,p)\in\Sigma_{\varepsilon}$. The \emph{basepoint
action} $p:C_{\mathsf{BN}}(T)\to C_{\mathsf{BN}}(T)$, which we also
denote as $p$ by abuse of notation, is the sum of the $f_{\varepsilon}$'s.

Let $c\in C$ be a crossing of $T$. For each pair $\varepsilon_{0},\varepsilon_{1}\in\{0,1\}^{C}$
such that they agree on $C\setminus\{c\}$ and $\varepsilon_{0}(c)=0$,
$\varepsilon_{1}(c)=1$, let $g_{\varepsilon_{1},\varepsilon_{0}}\in\mathrm{Hom}_{\mathsf{BN}_{\boldsymbol{e}}}(T_{\varepsilon_{1}},T_{\varepsilon_{0}})$
be the saddle cobordism corresponding to $c$. The \emph{basepoint
sliding homotopy} $H_{c}:C_{\mathsf{BN}}(T)\to C_{\mathsf{BN}}(T)$
is the sum of the $g_{\varepsilon_{1},\varepsilon_{0}}$'s.
\end{defn}

Note that $p:C_{\mathsf{BN}}(T)\to C_{\mathsf{BN}}(T)$ has bigrading
$(0,-2)$, and $H_{c}$ has bigrading $(-1,-2)$. Lemma~\ref{lem:homotopy-identities}
explains the name\emph{ basepoint sliding homotopy}.

\begin{figure}[h]
\begin{centering}
\includegraphics[scale=5]{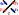}
\par\end{centering}
\caption{\label{fig:crossing}A crossing $c$ and points $p_{1},p_{2},q_{1},q_{2}$}
\end{figure}

\begin{lem}[{\cite[Lemma~2.3]{MR3190305}, \cite[Lemma~2.3]{MR3604486}, \cite[Theorem~4.2]{MR4521052}}]
\label{lem:homotopy-identities}Let $T\in\mathsf{Tang}_{\boldsymbol{e}}$,
let $c$ be a crossing of $T$, let points $p_{1},p_{2},q_{1},q_{2}$
be as in Figure~\ref{fig:crossing}, and let $r\in T$ be a point
away from the crossings. Then the following identities of maps $C_{\mathsf{BN}}(T)\to C_{\mathsf{BN}}(T)$
hold on the chain level (i.e.\ not merely up to homotopy):
\begin{gather*}
\partial H_{c}+H_{c}\partial=p_{1}+p_{2}=q_{1}+q_{2},\ H_{c}r=rH_{c},\ H_{c}^{2}=0,\\
H_{c}p_{1}=p_{2}H_{c},\ H_{c}p_{2}=p_{1}H_{c},\ H_{c}q_{1}=q_{2}H_{c},\ H_{c}q_{2}=q_{1}H_{c}.
\end{gather*}
\end{lem}

\section{\label{sec:Pointed-links-and}Pointed tangles and decorated cobordisms}

We consider more general kinds of decorations on the link cobordism
$\Sigma$ than in Section~\ref{sec:Introduction}; compare \cite{MR3905679}.
As a warmup, we first define the enhanced Khovanov TQFT $DKh$ for
links in Subsection~\ref{subsec:dkh-links}; this will serve as good
motivation for our definition of the enhanced Khovanov TQFT $\overline{D}_{\mathsf{BN}}$
for tangles in Subsection~\ref{subsec:dkh-tangles}. More precisely,
in this section we define them on the object level and for \emph{elementary
movies} (Definition~\ref{def:elementary-movies}); we prove well-definedness
in Section~\ref{sec:Invariance}.

First, let us set the stage. Let $\boldsymbol{X}$ be a finite set,
and let $R_{\boldsymbol{X}}:=\mathbb{F}[\boldsymbol{X}]/(\{x^{2}\}_{x\in\boldsymbol{X}})$:
i.e.\ if $\boldsymbol{X}=\{x_{1},\cdots,x_{n}\}$, then $R_{\boldsymbol{X}}:=\mathbb{F}[x_{1},\cdots,x_{n}]/(x_{1}^{2},\cdots,x_{n}^{2})$.
\begin{defn}[The categories $\mathsf{Tang}_{\boldsymbol{e},\boldsymbol{X}}$ and
$\mathsf{Link}_{\boldsymbol{X}}$]
Let $\boldsymbol{e}\subset D^{2}\times\{0\}\subset D^{2}\times[-1,1]=D^{3}$
be an oriented $0$-manifold. A \emph{generic $\boldsymbol{X}$-pointed
tangle with endpoints $\boldsymbol{e}$} is a tuple $(T,(\boldsymbol{p}_{x})_{x\in\boldsymbol{X}})$
of a generic oriented tangle $T\subset D^{3}$ with endpoints $\boldsymbol{e}$
and a pairwise disjoint collection $(\boldsymbol{p}_{x})_{x\in\boldsymbol{X}}$
of finite sets of points $\boldsymbol{p}_{x}\subset T$ away from
the crossings. Note that $\boldsymbol{p}_{x}$ may be empty. We call
elements $p\in\boldsymbol{p}_{x}$ \emph{basepoints}, and say that
\emph{$p$ has color $x$}. We denote $\vec{\boldsymbol{p}}:=(\boldsymbol{p}_{x})_{x\in\boldsymbol{X}}$
for simplicity.

An\emph{ $\boldsymbol{X}$-decorated tangle cobordism} from $(T,\vec{\boldsymbol{p}})$
to $(T',\vec{\boldsymbol{p}}')$ is a tuple $(\Sigma,(A_{x})_{x\in\boldsymbol{X}})$
of a tangle cobordism $\Sigma:T\to T'$ in $I\times D^{3}$ and a
collection $(A_{x})_{x\in\boldsymbol{X}}$ of properly immersed one-manifolds
$\varphi_{x}:A_{x}\to\Sigma$ that restrict to a bijection $\varphi_{x}|_{\partial A_{x}}:\partial A_{x}\to\boldsymbol{p}_{x}\sqcup\boldsymbol{p}_{x}'$.
Note that $A_{x}$ may be empty. We denote $\vec{A}:=(A_{x})_{x\in\boldsymbol{X}}$
for simplicity.

The category $\mathsf{Tang}_{\boldsymbol{e},\boldsymbol{X}}$ has
objects generic $\boldsymbol{X}$-pointed tangles with endpoints $\boldsymbol{e}$,
and morphisms $\boldsymbol{X}$-decorated tangle cobordisms where
we identify $(\Sigma,\vec{A})$ and $(\Sigma',\vec{A'})$ if $\Sigma$
and $\Sigma'$ are isotopic rel.\ $\partial$ in $I\times D^{3}$
and for all $x\in\boldsymbol{X}$ the mod $2$ homology classes $[A_{x},\partial A_{x}]$
and $[A_{x}',\partial A_{x}']$ are equal. The category $\mathsf{Link}_{\boldsymbol{X}}$
has objects generic $\boldsymbol{X}$-pointed tangles with no endpoints,
and morphisms $\boldsymbol{X}$-decorated tangle cobordisms where
we identify $(\Sigma,\vec{A})$ and $(\Sigma',\vec{A'})$ if $\Sigma$
and $\Sigma'$ are isotopic rel.\ $\partial$ in $I\times S^{3}$
and $[A_{x},\partial A_{x}]=[A_{x}',\partial A_{x}']$ for all $x\in\boldsymbol{X}$.

We often identify a morphism $(\Sigma,\vec{A})$ of $\mathsf{Tang}_{\boldsymbol{e},\boldsymbol{X}}$
or $\mathsf{Link}_{\boldsymbol{X}}$ with $(\Sigma,\vec{w})$ where
$\vec{w}=(w_{x})_{x\in\boldsymbol{X}}$ and $w_{x}=[A_{x},\partial A_{x}]\in H_{1}(\Sigma,\boldsymbol{p}_{x}\sqcup\boldsymbol{p}_{x}';\mathbb{F})$.
\end{defn}

Let $(\Sigma,\vec{A})$ be an $\boldsymbol{X}$-decorated tangle cobordism.
Recall that by applying a small isotopy to $\Sigma$ rel.\ $\partial$,
we can arrange $\Sigma$ to be a composition of\emph{ elementary undecorated
movies} (we add the adjective \emph{undecorated} to avoid confusion),
which consist of (i) cobordisms that come from isotopies of generic
tangles (we call these \emph{planar isotopies}), (ii) undecorated
Reidemeister movies, and (iii) undecorated Morse movies. Now, by applying
a small isotopy rel.\ $\partial$ to $\vec{A}$, we can arrange $(\Sigma,\vec{A})$
such that it is a composition of \emph{elementary movies}, defined
below.

\begin{figure}[h]
\begin{centering}
\includegraphics[scale=2.5]{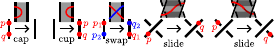}
\par\end{centering}
\caption{\label{fig:elementary-moves}Schematics of elementary movies. The
pairs of basepoints $(p,q)$, $(p_{1},q_{1})$, and $(p_{2},q_{2})$
are each the same color, but $p_{1}$ and $p_{2}$ need not have the
same color. The cap, cup, and swap movies are supported away from
the crossings.}
\end{figure}

We say that an object $Z$ is \emph{vertical} in $[a,b]\times D$
for some $D\subset D^{3}$ if $Z\cap([a,b]\times D)=[a,b]\times Y$
for some $Y$.
\begin{defn}[Elementary movies]
\label{def:elementary-movies}A \emph{Reidemeister} (resp.\ \emph{Morse})
\emph{movie} is a decorated tangle cobordism $(\Sigma,\vec{A})$ such
that $\Sigma$ is an undecorated Reidemeister (resp.\ Morse) movie,
and $\vec{A}$ is vertical and is disjoint from the region where the
undecorated movie happens.

A \emph{cap} (resp.\ \emph{cup}, \emph{swap}, \emph{slide}) movie
is a decorated tangle cobordism $(\Sigma,\vec{A})$ such that $\Sigma$
is vertical, $(\Sigma,\vec{A})$ agrees with the cap (resp.\ cup,
swap, slide) movie of Figure~\ref{fig:elementary-moves} in some
region, and $\vec{A}$ is vertical outside that region.

An \emph{elementary movie} is an $\boldsymbol{X}$-decorated tangle
cobordism $(\Sigma,\vec{A})$ that is one of the following:
\begin{enumerate*}
\item \label{enu:movie1}a tangle cobordism that comes from an isotopy of
generic $\boldsymbol{X}$-pointed tangles; we call these \emph{planar
isotopies},
\item \label{enu:movie2}a Reidemeister movie,
\item \label{enu:movie3}a Morse movie,
\item \label{enu:movie4}a cap movie,
\item \label{enu:movie5}a cup movie,
\item \label{enu:movie6}a swap movie, and
\item \label{enu:movie7}a slide movie.
\end{enumerate*}
\end{defn}

\subsection{\label{subsec:dkh-links}The enhanced Khovanov TQFT for links}

Let us define $DKh:\mathsf{Link}_{\boldsymbol{X}}\to D^{b}(\mathsf{Mod}_{R_{\boldsymbol{X}}})$.
\begin{defn}[$DKh$ on objects]
\label{def:dkh-objects}Let $(L,\vec{\boldsymbol{p}})$ be a generic
$\boldsymbol{X}$-pointed link. The $R_{\boldsymbol{X}}$-chain complex
$DKh(L,\vec{\boldsymbol{p}})$ is the chain complex $CKh(L)$ equipped
with the $R_{\boldsymbol{X}}$-module structure such that $x\in\boldsymbol{X}$
acts as $\sum_{p\in\boldsymbol{p}_{x}}p$.
\end{defn}

Hence, we are left to define maps $DKh(L,\vec{\boldsymbol{p}})\to DKh(L',\vec{\boldsymbol{p}}')$
in $D^{b}(\mathsf{Mod}_{R_{\boldsymbol{X}}})$ for the seven elementary
movies $(\Sigma,\vec{A})$. For (\ref{enu:movie2}) and (\ref{enu:movie3}),
the decoration $\vec{A}$ is vertical, and for (\ref{enu:movie4})-(\ref{enu:movie6}),
$\Sigma$ is vertical and $DKh(L,\vec{\boldsymbol{p}})=DKh(L',\vec{\boldsymbol{p}}')$.
Hence for (\ref{enu:movie1})-(\ref{enu:movie6}), the usual Khovanov
chain map $CKh(\Sigma):CKh(L)\to CKh(L')$ is $R_{\boldsymbol{X}}$-linear;
define the map $DKh(\Sigma,\vec{A})$ to be this map. Note that for
(\ref{enu:movie4})-(\ref{enu:movie6}), these maps are the identity
map.

Unlike (\ref{enu:movie1})-(\ref{enu:movie6}), the maps for slide
movies will only be defined as morphisms in $D^{b}(\mathsf{Mod}_{R_{\boldsymbol{X}}})$.
For $R_{\boldsymbol{X}}$-chain complexes $C,C'$, to define a morphism
$C\to C'$ in $D^{b}(\mathsf{Mod}_{R_{\boldsymbol{X}}})$, we work
with \emph{preferred free $R_{\boldsymbol{X}}$-resolutions} $q_{C}:\overline{C}\to C$
and $q_{C'}:\overline{C}'\to C'$, and we define an $R_{\boldsymbol{X}}$-chain
map $\overline{C}\to\overline{C}'$. Let us first define our\emph{
preferred free $R_{\boldsymbol{X}}$-resolution }of a general $R_{\boldsymbol{X}}$-chain
complex $C$.
\begin{example}[{Preferred free resolutions over $\mathbb{F}[x]/(x^{2})$}]
\label{exa:preferred-resolution}Let $C$ be a bounded above $\mathbb{F}$-chain
complex, and let $f:C\to C$ be a chain map such that $f^{2}=0$.
The map $f$ equips $C$ with an $\mathbb{F}[x]/(x^{2})$-module structure
where $x$ acts as $f$; denote as $C_{f}$ this $\mathbb{F}[x]/(x^{2})$-chain
complex. Then, our preferred $\mathbb{F}[x]/(x^{2})$-free resolution
of $C_{f}$ is 
\[\begin{tikzcd}[ampersand replacement=\&]
	\cdots \& {\xi ^{-2}C} \& {\xi ^{-1}C} \& C \&\& \\
	\cdots \& {q^{-2} \xi ^{-2}C} \& { q^{-2} \xi ^{-1} C} \& {q^{-2} C} \&\& \textcolor{rgb,255:red,92;green,92;blue,214}{{C_f}}
	\arrow["f"{description}, from=1-1, to=1-2]
	\arrow["1"{description}, from=1-1, to=2-2]
	\arrow["f"{description}, from=1-2, to=1-3]
	\arrow["1"{description}, from=1-2, to=2-3]
	\arrow["f"{description}, from=1-3, to=1-4]
	\arrow["1"{description}, from=1-3, to=2-4]
	\arrow[draw=none, from=1-4, to=2-4]
	\arrow["1"{description}, color={rgb,255:red,92;green,92;blue,214}, dashed, from=1-4, to=2-6]
	\arrow["f"{description}, from=2-1, to=2-2]
	\arrow["f"{description}, from=2-2, to=2-3]
	\arrow["f"{description}, from=2-3, to=2-4]
	\arrow["f"{description}, color={rgb,255:red,92;green,92;blue,214}, dashed, from=2-4, to=2-6]
\end{tikzcd}\]where $\xi$ is a formal variable with bidegree $(1,2)$, and $x$
acts by sending the first row to the second row via the identity.
The $\mathbb{F}[x]/(x^{2})$-linear quasi-isomorphism to $C_{f}$
is in blue.
\end{example}

We generalize Example\ \ref{exa:preferred-resolution} to chain complexes
over $R_{\boldsymbol{X}}$.
\begin{defn}
Let $\Xi_{x}:=\mathbb{F}[\xi_{x},\xi_{x}^{-1}]/\xi_{x}\mathbb{F}[\xi_{x}]$
where $\xi_{x}$ has bidegree $(1,2)$, and let $F_{x}$ be the free
$\mathbb{F}[x]/(x^{2})$-chain complex $(\mathbb{F}[x]/(x^{2}))\otimes_{\mathbb{F}}\Xi_{x}$
with differential $\partial(r\otimes\xi):=(xr)\otimes(\xi_{x}\xi)$,
i.e.\ 
\[
F_{x}=\cdots\xrightarrow{x}\xi_{x}^{-2}\mathbb{F}[x]/(x^{2})\xrightarrow{x}\xi_{x}^{-1}\mathbb{F}[x]/(x^{2})\xrightarrow{x}\mathbb{F}[x]/(x^{2}).
\]

Let $\Xi_{\boldsymbol{X}}:=\bigotimes_{x\in\boldsymbol{X}}\Xi_{x}$
and $F_{\boldsymbol{X}}:=\bigotimes_{x\in\boldsymbol{X}}F_{x}$ where
the tensor products are taken over $\mathbb{F}$. I.e.\ $F_{\boldsymbol{X}}=R_{\boldsymbol{X}}\otimes_{\mathbb{F}}\Xi_{\boldsymbol{X}}$
with differential $\partial(r\otimes\xi)=\sum_{x\in\boldsymbol{X}}(xr)\otimes(\xi_{x}\xi)$.
Note that $F_{\boldsymbol{X}}$ is a free $R_{\boldsymbol{X}}$-resolution
of $\mathbb{F}=R_{\boldsymbol{X}}/(\{x\}_{x\in\boldsymbol{X}})$.
\end{defn}

\begin{defn}[Preferred free $R_{\boldsymbol{X}}$-resolutions]
\label{def:preferred}The \emph{preferred free $R_{\boldsymbol{X}}$-resolution}
$\overline{C}$ of a bounded above $R_{\boldsymbol{X}}$-chain complex
$C$ is $\overline{C}:=C\otimes_{\mathbb{F}}R_{\boldsymbol{X}}\otimes_{\mathbb{F}}\Xi_{\boldsymbol{X}}$
where $R_{\boldsymbol{X}}$ acts on the middle $R_{\boldsymbol{X}}$,
and the differential is given by 
\[
\partial(c\otimes r\otimes\xi)=(\partial c)\otimes r\otimes\xi+\sum_{x\in\boldsymbol{X}}((xc)\otimes r+c\otimes(xr))\otimes(\xi_{x}\xi).
\]
The\emph{ preferred $R_{\boldsymbol{X}}$-quasi-isomorphism} (see
Lemma \ref{lem:quasi-iso}) $q_{C}:\overline{C}=C\otimes_{\mathbb{F}}R_{\boldsymbol{X}}\otimes_{\mathbb{F}}\Xi_{\boldsymbol{X}}\to C$
is given by $c\otimes r\otimes1\mapsto rc$ and $c\otimes r\otimes\xi\mapsto0$
for all monomials $\xi\neq1\in\Xi_{\boldsymbol{X}}$.

Define the functor $\mathsf{Free}:D^{b}(\mathsf{Mod}_{R_{\boldsymbol{X}}})\to D^{-}(\mathsf{Mod}_{R_{\boldsymbol{X}}})$
as $C\mapsto\overline{C}$ on objects; on morphisms, map $f:C\to D$
to $q_{D}^{-1}\circ f\circ q_{C}:\overline{C}\to\overline{D}$.
\end{defn}

\begin{rem}[Change of basis]
\label{rem:change-basis-1}The $R_{\boldsymbol{X}}$-chain complex
$\overline{C}$ is $R_{\boldsymbol{X}}$-isomorphic to the following
chain complex. Consider the usual $R_{\boldsymbol{X}}\otimes_{\mathbb{F}}R_{\boldsymbol{X}}$-chain
complex $C\otimes_{\mathbb{F}}F_{\boldsymbol{X}}$ with differential
$\partial_{C}\otimes\mathrm{Id}_{F_{\boldsymbol{X}}}+\mathrm{Id}_{C}\otimes\partial_{F_{\boldsymbol{X}}}$.
Now, consider the $\mathbb{F}$-algebra map $\Delta:R_{\boldsymbol{X}}\to R_{\boldsymbol{X}}\otimes_{\mathbb{F}}R_{\boldsymbol{X}}$
given by $x\mapsto1\otimes x+x\otimes1$ for $x\in\boldsymbol{X}$.
This makes $C\otimes_{\mathbb{F}}F_{\boldsymbol{X}}$ into an $R_{\boldsymbol{X}}$-chain
complex. We show that this chain complex is $R_{\boldsymbol{X}}$-isomorphic
to $\overline{C}$ in Lemma~\ref{lem:change-basis}, in a more general
setting.

The $R_{\boldsymbol{X}}$-quasi-isomorphism $C\otimes_{\mathbb{F}}F_{\boldsymbol{X}}\to C$
is as follows. View $\mathbb{F}$ as an $R_{\boldsymbol{X}}$-module
by identifying $R_{\boldsymbol{X}}/(\{x\}_{x\in\boldsymbol{X}})=\mathbb{F}$,
and consider the $R_{\boldsymbol{X}}$-quasi-isomorphism $F_{\boldsymbol{X}}\to\mathbb{F}$
given by $0$ on homological grading $\neq0$ and the quotient map
$R_{\boldsymbol{X}}\to\mathbb{F}$ on homological grading $0$. This
induces an $R_{\boldsymbol{X}}\otimes_{\mathbb{F}}R_{\boldsymbol{X}}$-quasi-isomorphism
$C\otimes_{\mathbb{F}}F_{\boldsymbol{X}}\to C\otimes_{\mathbb{F}}\mathbb{F}$.
Now, the $R_{\boldsymbol{X}}$-action on $C\otimes_{\mathbb{F}}\mathbb{F}=C$,
which is induced by $\Delta$ and the $R_{\boldsymbol{X}}\otimes_{\mathbb{F}}R_{\boldsymbol{X}}$-action
on $C\otimes_{\mathbb{F}}\mathbb{F}$, agrees with the $R_{\boldsymbol{X}}$-action
on $C$. Hence, the quasi-isomorphism $C\otimes_{\mathbb{F}}F_{\boldsymbol{X}}\to C\otimes_{\mathbb{F}}\mathbb{F}$
is $R_{\boldsymbol{X}}$-linear.
\end{rem}

\begin{lem}
\label{lem:quasi-iso}The above map $q_{C}:\overline{C}\to C$ is
an $R_{\boldsymbol{X}}$-quasi-isomorphism.
\end{lem}

\begin{proof}
It is clear that $q_{C}$ is $R_{\boldsymbol{X}}$-linear. To show
that it is a quasi-isomorphism, we induct on $|\boldsymbol{X}|$.
The base case is $\boldsymbol{X}=\{x\}$. Consider the filtration
of $\overline{C}=C\otimes_{\mathbb{F}}R_{x}\otimes_{\mathbb{F}}\Xi_{x}$
given by 
\[
F^{-n}\overline{C}:=\bigoplus_{0\le i\le n}C\otimes\mathbb{F}\otimes\xi_{x}^{-i}\mathbb{F}\oplus\bigoplus_{0\le i\le n-1}C\otimes x\mathbb{F}\otimes\xi_{x}^{-i}\mathbb{F}.
\]
Let $M:=\mathrm{Cone}(\overline{C}\xrightarrow{q_{C}}C)$, and consider
the filtration $0:=F^{1}M\subset F^{0}M\subset F^{-1}M\subset\cdots$
given by $F^{-n}M:=(F^{-n}\overline{C})\oplus C$ for $n\ge0$. For
each $n\ge0$, the chain complex $F^{-n}M/F^{-n+1}M$ is $\mathbb{F}$-isomorphic
to $\mathrm{Cone}(C\xrightarrow{\mathrm{Id}}C)$, and hence has trivial
homology. Thus $M$ also has trivial homology since for each $h\in\mathbb{Z}$,
there exists some $n=n(h)$ such that the homological grading $h$
summand of $M$ is contained in $F^{n}M$.

For the induction step, let $\boldsymbol{X}=\boldsymbol{Y}\sqcup\{x\}$,
and let $q_{C,\boldsymbol{Y}}:\overline{C}_{\boldsymbol{Y}}\to C$
(resp.\ $q_{C,\boldsymbol{X}}:\overline{C}_{\boldsymbol{X}}\to C$)
be the preferred free $R_{\boldsymbol{Y}}$- (resp.\ $R_{\boldsymbol{X}}$-)resolution
of $C$ viewed as an $R_{\boldsymbol{Y}}$- (resp.\ $R_{\boldsymbol{X}}$-)chain
complex. Then $\overline{C}_{\boldsymbol{Y}}$ is an $R_{x}$-chain
complex; let $q_{C,\boldsymbol{Y},x}:\overline{C}_{\boldsymbol{Y},x}\to\overline{C}_{\boldsymbol{Y}}$
be its preferred $R_{x}$-resolution. Then, the identifications $R_{\boldsymbol{X}}\cong R_{\boldsymbol{Y}}\otimes_{\mathbb{F}}R_{x}$
and $\Xi_{\boldsymbol{X}}\cong\Xi_{\boldsymbol{Y}}\otimes_{\mathbb{F}}\Xi_{x}$
induce an $R_{\boldsymbol{X}}$-isomorphism between $\overline{C}_{\boldsymbol{Y},x}=(C\otimes R_{\boldsymbol{Y}}\otimes\Xi_{\boldsymbol{Y}})\otimes R_{x}\otimes\Xi_{x}$
and $\overline{C}_{\boldsymbol{X}}=C\otimes R_{\boldsymbol{X}}\otimes\Xi_{\boldsymbol{X}}$
which is an isomorphism of chain complexes. Moreover, under this isomorphism,
$q_{C,\boldsymbol{Y}}\circ q_{C,\boldsymbol{Y},x}$ agrees with $q_{C,\boldsymbol{X}}$.
Hence $q_{C,\boldsymbol{X}}$ is a quasi-isomorphism.
\end{proof}
\begin{rem}[The image of $q_{C}^{-1}$ in $K^{-}(\mathsf{Mod}_{\mathbb{F}})$]
\label{rem:f-linear-inverse}The image of $q_{C}^{-1}\in\mathrm{Hom}_{D^{-}(\mathsf{Mod}_{R_{\boldsymbol{X}}})}(C,\overline{C})$
under the forgetful functor $D^{-}(\mathsf{Mod}_{R_{\boldsymbol{X}}})\to K^{-}(\mathsf{Mod}_{\mathbb{F}})$
is the $\mathbb{F}$-linear chain homotopy inverse of $q_{C}:\overline{C}\to C$
(unique up to $\mathbb{F}$-chain homotopy). Explicitly, let $f_{C}:C\to\overline{C}:c\mapsto c\otimes1\otimes1$.
Then we have $q_{C}\circ f_{C}=\mathrm{Id}_{C}$, and so $f_{C}$
is the image of $q_{C}^{-1}$ in $K^{-}(\mathsf{Mod}_{\mathbb{F}})$.
\end{rem}

\begin{defn}[$\overline{D}Kh$]
Denote as $\overline{D}Kh(L,\vec{\boldsymbol{p}})$ our preferred
free $R_{\boldsymbol{X}}$-resolution $DKh(L,\vec{\boldsymbol{p}})\otimes_{\mathbb{F}}R_{\boldsymbol{X}}\otimes_{\mathbb{F}}\Xi_{\boldsymbol{X}}$
of $DKh(L,\vec{\boldsymbol{p}})$.
\end{defn}

Let $(\Sigma,\vec{A}):(L,\vec{\boldsymbol{p}})\to(L,\vec{\boldsymbol{p}}')$
be a slide movie. Let the basepoint that slides have color $y\in\boldsymbol{X}$,
and let $H$ be the basepoint sliding homotopy. Then, define 
\[
\overline{D}Kh(\Sigma,\vec{A}):\overline{D}Kh(L,\vec{\boldsymbol{p}})\to\overline{D}Kh(L,\vec{\boldsymbol{p}}'):c\otimes r\otimes\xi\mapsto c\otimes r\otimes\xi+(Hc)\otimes r\otimes(\xi_{y}\xi).
\]
To avoid confusion, for $x\in\boldsymbol{X}$, denote as $x$ (resp.\ $x'$)
its action on $CKh(L)$ (resp.\ $CKh(L')$). That $\overline{D}Kh(\Sigma,\vec{A})$
is $R_{\boldsymbol{X}}$-linear is clear from the definition; that
it is a chain map follows from Lemma~\ref{lem:homotopy-identities}
and more precisely the following identities of maps $CKh(L)\to CKh(L)$:
$y+y'=H\partial+\partial H$, for all $x$ we have $Hx=x'H$, and
for all $x\neq y$ we have $x=x'$.

Define $DKh(\Sigma,\vec{A})\in\mathrm{Hom}_{D^{b}(\mathsf{Mod}_{R_{\boldsymbol{X}}})}(DKh(L,\vec{\boldsymbol{p}}),DKh(L,\vec{\boldsymbol{p}}'))$
as the following composition:
\[
DKh(\Sigma,\vec{A}):DKh(L,\vec{\boldsymbol{p}})\xrightarrow{q_{DKh(L,\vec{\boldsymbol{p}})}^{-1}}\overline{D}Kh(L,\vec{\boldsymbol{p}})\xrightarrow{\overline{D}Kh(\Sigma,\vec{A})}\overline{D}Kh(L,\vec{\boldsymbol{p}}')\xrightarrow{q_{DKh(L,\vec{\boldsymbol{p}}')}}DKh(L,\vec{\boldsymbol{p}}').
\]
Note that although $\overline{D}Kh(L,\vec{\boldsymbol{p}})$ and $\overline{D}Kh(L,\vec{\boldsymbol{p}}')$
only belong to $D^{-}(\mathsf{Mod}_{R_{\boldsymbol{X}}})$, since
$D^{b}(\mathsf{Mod}_{R_{\boldsymbol{X}}})$ is a full subcategory
of $D^{-}(\mathsf{Mod}_{R_{\boldsymbol{X}}})$, the above composition
lies in $D^{b}(\mathsf{Mod}_{R_{\boldsymbol{X}}})$.

Finally, for completeness, let us define $\overline{D}Kh(\Sigma,\vec{A}):\overline{D}Kh(L,\vec{\boldsymbol{p}})\to\overline{D}Kh(L,\vec{\boldsymbol{p}}')$
also for the elementary movies (\ref{enu:movie1})-(\ref{enu:movie6}):
for these elementary movies, define $\overline{D}Kh(\Sigma,\vec{A}):=CKh(\Sigma)\otimes\mathrm{Id}_{R_{\boldsymbol{X}}\otimes_{\mathbb{F}}\Xi_{\boldsymbol{X}}}$.

\subsection{\label{subsec:dkh-tangles}The enhanced Khovanov TQFT for tangles}

In this section, we define the enhanced Khovanov TQFT for tangles.
Instead of trying to ``derive'' the Bar-Natan category $\mathsf{BN}_{\boldsymbol{e}}$,
we will exclusively work with chain complexes that resemble the preferred
free resolutions from Subsection~\ref{subsec:dkh-links}. Indeed,
the target category is $K_{R_{\boldsymbol{X}}}^{-}(\mathsf{BN}_{\boldsymbol{e}})$
which we now define; the functor will be called $\overline{D}_{\mathsf{BN}}:\mathsf{Tang}_{\boldsymbol{e},\boldsymbol{X}}\to K_{R_{\boldsymbol{X}}}^{-}(\mathsf{BN}_{\boldsymbol{e}})$.

\begin{defn}
Let $\mathsf{A}$ be an additive $\mathbb{F}$-linear category, and
let $R$ be a ring. Define the \emph{homotopy category of bounded
above $R$-chain complexes} $K_{R}^{-}(\mathsf{A})$ as follows.

An object of $K_{R}^{-}(\mathsf{A})$ is a pair $(C,\rho_{C})$ of
a bounded above chain complex $C\in Ch^{-}(\mathsf{A})$ and a ring
homomorphism $\rho_{C}:R\to\mathrm{Hom}_{Ch^{-}(\mathsf{A})}(C,C)$,
called the \emph{$R$-action} on $C$. Given two such pairs $(C,\rho_{C})$
and $(D,\rho_{D})$, a map $f:C\to D$ is \emph{$R$-linear }if for
all $r\in R$, $f\circ\rho_{C}(r)=\rho_{D}(r)\circ f$. The morphism
space $\mathrm{Hom}_{K_{R}^{-}(\mathsf{A})}((C,\rho_{C}),(D,\rho_{D}))$
is the $R$-module of $R$-linear chain maps modulo $R$-linear chain
homotopies. For simplicity, we often write $C$ instead of $(C,\rho_{C})$.
\end{defn}

\begin{defn}
\label{def:preferred-resolution}Given an additive $\mathbb{F}$-linear
category $\mathsf{A}$ and a chain complex $C\in Ch^{-}(\mathsf{A})$
together with an $R_{\boldsymbol{X}}$-action $\rho_{C}:R_{\boldsymbol{X}}\to\mathrm{Hom}_{Ch^{-}(\mathsf{A})}(C,C)$,
define its \emph{preferred free $R_{\boldsymbol{X}}$-resolution}
$(\overline{C},\rho_{\overline{C}})\in K_{R_{\boldsymbol{X}}}^{-}(\mathsf{A})$
as $\overline{C}:=C\otimes_{\mathbb{F}}(R_{\boldsymbol{X}}\otimes_{\mathbb{F}}\Xi_{\boldsymbol{X}})$
with differential 
\[
\partial_{\overline{C}}:=\partial_{C}\otimes\mathrm{Id}_{R_{\boldsymbol{X}}}\otimes\mathrm{Id}_{\Xi_{\boldsymbol{X}}}+\sum_{x\in\boldsymbol{X}}(\rho_{C}(x)\otimes\mathrm{Id}_{R_{\boldsymbol{X}}}+\mathrm{Id}_{C}\otimes x)\otimes\xi_{x},
\]
where $x:R_{\boldsymbol{X}}\to R_{\boldsymbol{X}}$ (resp.\ $\xi_{x}:\Xi_{\boldsymbol{X}}\to\Xi_{\boldsymbol{X}}$)
means multiplication by $x$ (resp.\ $\xi_{x}$) (e.g.\ $\partial_{F_{\boldsymbol{X}}}=\sum_{x\in\boldsymbol{X}}x\otimes\xi_{x}$)
and the $R_{\boldsymbol{X}}$-action $\rho_{\overline{C}}$ is given
by acting on $R_{\boldsymbol{X}}$, i.e.\ $\rho_{\overline{C}}(r):={\rm Id}_{C}\otimes r\otimes{\rm Id}_{\Xi_{{\bf X}}}$.
\end{defn}

\begin{defn}[The tangle TQFT $\overline{D}_{\mathsf{BN}}$]
\label{def:dbarbn}Let $(T,\vec{\boldsymbol{p}})$ be a generic $\boldsymbol{X}$-pointed
$\boldsymbol{e}$-tangle. Consider the $R_{\boldsymbol{X}}$-action
on $C_{\mathsf{BN}}(T)$ given by $\rho(x):=\sum_{p\in\boldsymbol{p}_{x}}p$;
let $\overline{D}_{\mathsf{BN}}(T,\vec{\boldsymbol{p}}):=C_{\mathsf{BN}}(T)\otimes_{\mathbb{F}}(R_{\boldsymbol{X}}\otimes_{\mathbb{F}}\Xi_{\boldsymbol{X}})$
be its preferred free $R_{\boldsymbol{X}}$-resolution.

If $(\Sigma,\vec{A}):(T,\vec{\boldsymbol{p}})\to(T',\vec{\boldsymbol{p}}')$
is one of the elementary movies (\ref{enu:movie1})-(\ref{enu:movie6}),
let 
\[
\overline{D}_{\mathsf{BN}}(\Sigma,\vec{A}):=C_{\mathsf{BN}}(\Sigma)\otimes\mathrm{Id}_{R_{\boldsymbol{X}}\otimes\Xi_{\boldsymbol{X}}}:\overline{D}_{\mathsf{BN}}(T,\vec{\boldsymbol{p}})\to\overline{D}_{\mathsf{BN}}(T',\vec{\boldsymbol{p}}').
\]
For (\ref{enu:movie7}), we have $T=T'$; let the relevant basepoints
have color $y$, and let $H:C_{\mathsf{BN}}(T)\to C_{\mathsf{BN}}(T)$
be the corresponding basepoint sliding homotopy. Let 
\[
\overline{D}_{\mathsf{BN}}(\Sigma,\vec{A}):=\mathrm{Id}_{C_{\mathsf{BN}}(T)}\otimes\mathrm{Id}_{R_{\boldsymbol{X}}\otimes\Xi_{\boldsymbol{X}}}+H\otimes\mathrm{Id}_{R_{\boldsymbol{X}}}\otimes\xi_{y}:\overline{D}_{\mathsf{BN}}(T,\vec{\boldsymbol{p}})\to\overline{D}_{\mathsf{BN}}(T,\vec{\boldsymbol{p}}').
\]
\end{defn}

\begin{lem}[Recovers $DKh$]
The functor $\overline{D}_{\mathsf{BN}}$ recovers the functor $DKh$
from Subsection~\ref{subsec:dkh-links}. More precisely, let $\mathsf{Free}$
be the functor from Definition~\ref{def:preferred}, and let ${\cal F}:K_{R_{\boldsymbol{X}}}^{-}(\mathsf{BN}_{\boldsymbol{e}})\to K^{-}(\mathsf{Mod}_{R_{\boldsymbol{X}}})$
be the functor induced by the usual TQFT ${\cal F}$ from \cite[Section~7]{BN1}.
The following commutes: 
\[\begin{tikzcd}[ampersand replacement=\&]
	{\mathsf{Tang}_{\emptyset,\boldsymbol{X}}} \& {K_{R_{\boldsymbol{X}}}^{-}(\mathsf{BN}_{\emptyset })} \& {K^{-}(\mathsf{Mod}_{R_{\boldsymbol{X}}})} \& {D^{-}(\mathsf{Mod}_{R_{\boldsymbol{X}}})} \\
	{\mathsf{Link}_{\boldsymbol{X}}} \&\&\& {D^b (\mathsf{Mod}_{R_{\boldsymbol{X}}})}
	\arrow["{\overline{D}_{\mathsf{BN}}}", from=1-1, to=1-2]
	\arrow[from=1-1, to=2-1]
	\arrow["{\cal F}", from=1-2, to=1-3]
	\arrow[from=1-3, to=1-4]
	\arrow["DKh", from=2-1, to=2-4]
	\arrow["{\mathsf{Free}}"', from=2-4, to=1-4]
\end{tikzcd}\]
\end{lem}

\begin{proof}
Clear from the definitions.
\end{proof}
The following lemma in particular shows that the isomorphism type
of $\overline{D}_{\mathsf{BN}}(T,\vec{\boldsymbol{p}})$ is invariant
under isotopies of $(T,\vec{\boldsymbol{p}})$; compare \cite[Proposition~2.9]{MR3604486},
\cite[Theorem~4.2~(1)]{MR4521052}.
\begin{lem}
\label{lem:homotopy-equivalence}Let $(\Sigma,\vec{A}):(T,\vec{\boldsymbol{p}})\to(T',\vec{\boldsymbol{p}'})$
be such that $\Sigma$ is induced by an isotopy $T\to T'$. Then $\overline{D}_{\mathsf{BN}}(\Sigma,\vec{A})$
has bidegree $(0,0)$, and it has a bidegree $(0,0)$ $R_{\boldsymbol{X}}$-homotopy
inverse.
\end{lem}

\begin{proof}
It is sufficient to show the lemma for the cases where $(\Sigma,\vec{A})$
is a Reidemeister, cap, cup, swap, or slide movie. First, the cup,
cap, and swap maps are the identity. If $(\Sigma,\vec{A})$ is a Reidemeister
movie, let $\Sigma'$ be the inverse undecorated Reidemeister movie.
Then $C_{\mathsf{BN}}(\Sigma)$ and $C_{\mathsf{BN}}(\Sigma')$ are
bidegree $(0,0)$ homotopy inverses. Recall that $\overline{D}_{\mathsf{BN}}(\Sigma):=C_{\mathsf{BN}}(\Sigma)\otimes\mathrm{Id}_{R_{\boldsymbol{X}}\otimes\Xi_{\boldsymbol{X}}}$
and $\overline{D}_{\mathsf{BN}}(\Sigma'):=C_{\mathsf{BN}}(\Sigma')\otimes\mathrm{Id}_{R_{\boldsymbol{X}}\otimes\Xi_{\boldsymbol{X}}}$;
hence $\overline{D}_{\mathsf{BN}}(\Sigma)$ and $\overline{D}_{\mathsf{BN}}(\Sigma')$
are bidegree $(0,0)$ $R_{\boldsymbol{X}}$-homotopy inverses.

If $(\Sigma,\vec{A})$ is a slide movie, then let $(\Sigma',\vec{A'}):(T',\vec{\boldsymbol{p}'})\to(T,\vec{\boldsymbol{p}})$
be the inverse slide movie; $\overline{D}_{\mathsf{BN}}(\Sigma,\vec{A})$
and $\overline{D}_{\mathsf{BN}}(\Sigma',\vec{A'})$ are strict inverses
on the chain level. (Here, we use Lemma~\ref{lem:homotopy-identities}:
$H^{2}=0$ if $H$ is the basepoint sliding homotopy.)
\end{proof}

\section{\label{sec:Invariance}Invariance}

In this section, we show that the functors $\overline{D}_{\mathsf{BN}}$
and $DKh$ from Section~\ref{sec:Pointed-links-and} are well-defined.
Let us first recall the steps that go into showing that $C_{\mathsf{BN}}(\Sigma)$
and $CKh$ are well-defined \cite{Ja,BN1,MR2171235,MWW} (compare
\cite{MR4584595}).

\begin{enumerate}
[label=(Step~\arabic*)]
\item \label{enu:step1}Define $C_{\mathsf{BN}}(\Sigma)$ for elementary
undecorated movies, and hence for a sequence of elementary undecorated
movies.
\end{enumerate}
Let us introduce the following definition for \ref{enu:step2}.
\begin{defn}[Generic undecorated tangle cobordisms]
\label{def:tang-gen}An undecorated tangle cobordism $\Sigma\subset I\times D^{3}=I\times D^{2}\times[-1,1]$
is \emph{generic} if the composition $f:\Sigma\to I\times D^{2}\times[-1,1]\to I\times D^{2}$
is generic \cite[Section~1.2]{CS} and the height function $\pi:I\times D^{2}\to I$
is also generic \cite[Definition~1.2]{CS}. The \emph{Reidemeister
loci }are points on $\Sigma$ that correspond to critical points,
branch points, and triple points of $\pi\circ f_{2}$ (recall the
definition of $f_{2}$ from \cite[Section~1.2]{CS}), and the \emph{Morse
loci} are points on $\Sigma$ that correspond to critical points of
$\pi\circ f$.

Let $\mathsf{Tang}_{\boldsymbol{e}}^{\mathrm{gen}}$ be the category
whose objects are undecorated tangles with endpoints $\boldsymbol{e}$
and morphisms are generic undecorated tangle cobordisms, where two
such cobordisms are identified if they are isotopic rel.~$\partial$
in $I\times D^{3}$ through generic undecorated tangle cobordisms.
\end{defn}

\begin{enumerate}
[label=(Step~\arabic*)]\addtocounter{enumi}{1}
\item \label{enu:step2}For sequences $\Sigma_{1}$ and $\Sigma_{2}$ of
elementary undecorated movies, show that $C_{\mathsf{BN}}(\Sigma_{1})\sim C_{\mathsf{BN}}(\Sigma_{2})$
if the undecorated tangle cobordisms that correspond to $\Sigma_{1}$
and $\Sigma_{2}$ are isotopic rel.~$\partial$ in $I\times D^{3}$
through generic undecorated tangle cobordisms. Hence, $C_{\mathsf{BN}}:\mathsf{Tang}_{\boldsymbol{e}}^{\mathrm{gen}}\to K^{b}(\mathsf{BN}_{\boldsymbol{e}})$
is well-defined.
\item \label{enu:step3}For generic undecorated tangle cobordisms $\Sigma_{1}$
and $\Sigma_{2}$, show that $C_{\mathsf{BN}}(\Sigma_{1})\sim C_{\mathsf{BN}}(\Sigma_{2})$
if $\Sigma_{1}$ and $\Sigma_{2}$ are isotopic rel.~$\partial$
in $I\times D^{3}$.
\item \label{enu:step4}For $\boldsymbol{e}=\emptyset$, show that $C_{\mathsf{BN}}(\Sigma_{1})\sim C_{\mathsf{BN}}(\Sigma_{2})$
if $\Sigma_{1}$ and $\Sigma_{2}$ are isotopic rel.~$\partial$
in $I\times S^{3}$.
\end{enumerate}

For \ref{enu:step2}, one checks that the maps defined in \ref{enu:step1}
are invariant under (1) replacing a pair of planar isotopies with
their composition, and (2) commuting a planar isotopy and a Reidemeister
or Morse movie, both of which are clear from the definitions. For
\ref{enu:step3}, one uses \cite{CS} and checks that $C_{\mathsf{BN}}$
is invariant under (1) swapping the order of distant Reidemeister
or Morse movies that are adjacent in time, which we call \emph{far
commutation}, and (2) movie moves. Finally for \ref{enu:step4}, one
checks that $C_{\mathsf{BN}}$ is invariant under certain global moves
\cite[Formula~(3-1)]{MWW}; we call them the \emph{undecorated sweep-around
moves}.

We will follow these four steps for $\overline{D}_{\mathsf{BN}}$
and $\boldsymbol{X}$-decorated tangle cobordisms. We have done \ref{enu:step1}
in Subsection~\ref{subsec:dkh-tangles}; for \ref{enu:step2} we
introduce the following definition.
\begin{defn}[Generic $\boldsymbol{X}$-decorated tangle cobordisms]
Let $\Sigma$ be a generic undecorated tangle cobordism. Let $Cr\subset\Sigma$
be the subset that corresponds to the crossings. Away from the Reidemeister
and Morse loci, $Cr$ is a properly embedded $1$-manifold. For each
$x\in\boldsymbol{X}$, let $A_{x}$ be a $1$-manifold, let $A:=\bigsqcup_{x\in\boldsymbol{X}}A_{x}$,
and let $\varphi:A\to\Sigma$ be a smooth map such that $\varphi(\partial A)\subset\partial\Sigma$
and $\varphi$ is a proper embedding near $\partial A$. 

We say $(\Sigma,\vec{A})$ is a \emph{generic $\boldsymbol{X}$-decorated
tangle cobordism} if $\varphi$ satisfies the following conditions:
\begin{enumerate}
[label=(Cond~\Alph*)]
\item \label{enu:conda}$\varphi:A\to\Sigma$ is an immersion.
\item \label{enu:condb}Self-intersections are transverse, and there are
no triple self-intersections. Call self-intersections the \emph{swap
loci}.
\item The composition $A\xrightarrow{\varphi}I\times D^{3}\xrightarrow{\mathrm{proj}}I$
is Morse. Call the local maxima the \emph{cap loci} and the local
minima the \emph{cup loci}.
\item \label{enu:condd}$A$ intersects $Cr$ transversely. Call these the
\emph{slide loci}.
\end{enumerate}
\begin{enumerate}
[label=(Conds~E-J)]
\item \label{enu:conde-j}The Reidemeister, Morse, cap, cup, swap, and
slide loci are pairwise distinct and are supported in pairwise distinct
times.\footnote{In Subsection~\ref{subsec:Invariance-under-changing}, we will split
\ref{enu:conde-j} into six conditions \ref{enu:conde}-\ref{enu:condj}.}
\end{enumerate}
Define the category $\mathsf{Tang}_{\boldsymbol{e},\boldsymbol{X}}^{\mathrm{gen}}$
analogously to Definition~\ref{def:tang-gen}.
\end{defn}

Given this definition, \ref{enu:step2} is immediate just like the
undecorated case: we check that $\overline{D}_{\mathsf{BN}}$ is invariant
under replacing a pair of planar isotopies with their composition
and commuting a planar isotopy and any of the other elementary movies
(\ref{enu:movie2})-(\ref{enu:movie7}). Also note that any generic
$\boldsymbol{X}$-decorated tangle cobordism is isotopic rel.\ $\partial$
through generic $\boldsymbol{X}$-decorated tangle cobordisms to a
composition of elementary movies; hence $\overline{D}_{\mathsf{BN}}:\mathsf{Tang}_{\boldsymbol{e},\boldsymbol{X}}^{\mathrm{gen}}\to K_{R_{\boldsymbol{X}}}^{-}(\mathsf{BN}_{\boldsymbol{e}})$
is well-defined.

We carry out \ref{enu:step3} for $\overline{D}_{\mathsf{BN}}$ in
Subsections~\ref{subsec:Collapsing-colors}-\ref{subsec:The-general-case},
and \ref{enu:step4} for $\overline{D}_{\mathsf{BN}}$ in Subsection~\ref{subsec:The-sweep-around-move}.

\subsection{\label{subsec:Collapsing-colors}Collapsing colors}

In the undecorated case \cite{BN1}, \ref{enu:step3} and \ref{enu:step4}
are checked locally. In this subsection we prove Proposition~\ref{prop:check-movie-moves-local},
which lets us check the decorated movie moves locally. First, let
us observe that we can compose tangles and cobordisms (Proposition~\ref{prop:composing-tangles-cobordisms})
similarly to the undecorated case (\cite{BN1}, Subsection~\ref{subsec:Bar-Natan's-tangle-invariant}).
\begin{prop}[Composing tangles and cobordisms]
\label{prop:composing-tangles-cobordisms}Let $D$ be an oriented
planar arc diagram with input endpoints $\boldsymbol{e}_{1},\cdots,\boldsymbol{e}_{d}$
and output endpoints $\boldsymbol{e}$, let $\boldsymbol{X}_{1},\cdots,\boldsymbol{X}_{d}$
be finite sets, and let $\boldsymbol{X}=\boldsymbol{X}_{1}\sqcup\cdots\sqcup\boldsymbol{X}_{d}$.
Then, $D$ induces a functor 
\[
\mathsf{Tang}_{\boldsymbol{e}_{1},\boldsymbol{X}_{1}}^{\mathrm{gen}}\times\cdots\times\mathsf{Tang}_{\boldsymbol{e}_{d},\boldsymbol{X}_{d}}^{\mathrm{gen}}\to\mathsf{Tang}_{\boldsymbol{e},\boldsymbol{X}}^{\mathrm{gen}}.
\]

Given $R_{\boldsymbol{X}_{i}}$-chain complexes $(C_{i},\rho_{i})\in K_{R_{\boldsymbol{X}_{i}}}^{-}(\mathsf{BN}_{\boldsymbol{e}_{i}})$,
let $C:=C_{1}\otimes\cdots\otimes C_{d}\in Ch^{-}(\mathsf{BN}_{\boldsymbol{e}})$
and let $\rho:R_{\boldsymbol{X}}\to\mathrm{Hom}_{Ch^{-}(\mathsf{BN}_{\boldsymbol{e}})}(C,C)$
be such that $\rho(x_{i}):=\mathrm{Id}_{C_{1}}\otimes\cdots\otimes\rho_{i}(x_{i})\otimes\cdots\otimes\mathrm{Id}_{C_{d}}$
for $x_{i}\in\boldsymbol{X}_{i}$. This induces a functor
\[
K_{R_{\boldsymbol{X}_{1}}}^{-}(\mathsf{BN}_{\boldsymbol{e}_{1}})\times\cdots\times K_{R_{\boldsymbol{X}_{d}}}^{-}(\mathsf{BN}_{\boldsymbol{e}_{d}})\to K_{R_{\boldsymbol{X}}}^{-}(\mathsf{BN}_{\boldsymbol{e}}).
\]

These functors are associative and commute with $\overline{D}_{\mathsf{BN}}$.
\end{prop}

\begin{proof}
Clear from the definitions.
\end{proof}
\begin{prop}[Can check movie moves locally]
\label{prop:check-movie-moves-local}Let $D$ be an oriented planar
arc diagram with input endpoints $\boldsymbol{e}_{1},\cdots,\boldsymbol{e}_{d}$
and output endpoints $\boldsymbol{e}$. For $i=1,\cdots,d$, let $\boldsymbol{X}_{i}$
be a finite set, and let $(\Sigma_{i},\vec{A_{i}}),(\Sigma_{i}',\vec{A_{i}'}):(T_{i},\vec{\boldsymbol{p}_{i}})\to(T_{i}',\vec{\boldsymbol{p}_{i}'})$
be morphisms in $\mathsf{Tang}_{\boldsymbol{e}_{i},\boldsymbol{X}_{i}}^{\mathrm{gen}}$.
Let $\boldsymbol{X}=\bigcup_{i=1}^{d}\boldsymbol{X}_{i}$ and let
$(\Sigma,\vec{A}),(\Sigma',\vec{A'})\in\mathrm{Mor}_{\mathsf{Tang}_{\boldsymbol{e},\boldsymbol{X}}^{\mathrm{gen}}}((T,\vec{\boldsymbol{p}}),(T,\vec{\boldsymbol{p}'}))$
be the composition of the $(\Sigma_{i},\vec{A_{i}})$'s and $(\Sigma_{i}',\vec{A_{i}'})$'s,
respectively. If for all $i=1,\cdots d$ $\overline{D}_{\mathsf{BN}}(\Sigma_{i},\vec{A_{i}})$
and $\overline{D}_{\mathsf{BN}}(\Sigma_{i}',\vec{A_{i}'})$ are $R_{\boldsymbol{X}_{i}}$-chain
homotopic, then $\overline{D}_{\mathsf{BN}}(\Sigma,\vec{A})$ and
$\overline{D}_{\mathsf{BN}}(\Sigma',\vec{A}')$ are $R_{\boldsymbol{X}}$-chain
homotopic.
\end{prop}

Of course, Proposition~\ref{prop:check-movie-moves-local} immediately
follows from Proposition~\ref{prop:composing-tangles-cobordisms}
if the $\boldsymbol{X}_{i}$'s are pairwise disjoint. We will show
that, roughly speaking, if $\boldsymbol{W},\boldsymbol{X}$ are finite
sets such that $|\boldsymbol{W}|\ge|\boldsymbol{X}|$, then studying
$\overline{D}_{\mathsf{BN}}:\mathsf{Tang}_{\boldsymbol{e},\boldsymbol{W}}^{\mathrm{gen}}\to K_{R_{\boldsymbol{W}}}^{-}(\mathsf{BN}_{\boldsymbol{e}})$
is enough to study $\overline{D}_{\mathsf{BN}}:\mathsf{Tang}_{\boldsymbol{e},\boldsymbol{X}}^{\mathrm{gen}}\to K_{R_{\boldsymbol{X}}}^{-}(\mathsf{BN}_{\boldsymbol{e}})$.
Let us fix a surjective function $\sigma:\boldsymbol{W}\to\boldsymbol{X}$.
\begin{defn}
\label{def:sigma-color}Let $R_{\sigma}:R_{\boldsymbol{X}}\to R_{\boldsymbol{W}}$
be the ring homomorphism such that $x\mapsto\sum_{\sigma(w)=x}w$
for $x\in\boldsymbol{X}$. Let $\Xi_{\sigma}:\Xi_{\boldsymbol{X}}\to\Xi_{\boldsymbol{W}}$
be the $\mathbb{F}$-linear map such that $\prod_{x\in\boldsymbol{X}}\xi_{x}^{-n_{x}}\mapsto\sum\prod_{w\in\boldsymbol{W}}\xi_{w}^{-n_{w}}$
where the sum is taken over $(n_{w})_{w\in\boldsymbol{W}}$ such that
for all $x\in\boldsymbol{X}$, $n_{x}=\sum_{\sigma(w)=x}n_{w}$.

Let $\mathsf{A}$ be an additive $\mathbb{F}$-linear category, let
$(C,\rho)\in K_{R_{\boldsymbol{W}}}^{-}(\mathsf{A})$, and let $(\overline{C}_{\boldsymbol{W}},\rho_{\boldsymbol{W}})\in K_{R_{\boldsymbol{W}}}^{-}(\mathsf{A})$
be its preferred free $R_{\boldsymbol{W}}$-resolution. Note that
$(C,\rho\circ R_{\sigma})\in K_{R_{\boldsymbol{X}}}^{-}(\mathsf{A})$;
denote as $(\overline{C}_{\boldsymbol{X}},\rho_{\boldsymbol{X}})\in K_{R_{\boldsymbol{X}}}^{-}(\mathsf{A})$
its preferred free $R_{\boldsymbol{X}}$-resolution. Define 
\[
\overline{\sigma}:=\mathrm{Id}_{C}\otimes R_{\sigma}\otimes\Xi_{\sigma}:C\otimes_{\mathbb{F}}R_{\boldsymbol{X}}\otimes_{\mathbb{F}}\Xi_{\boldsymbol{X}}\to C\otimes_{\mathbb{F}}R_{\boldsymbol{W}}\otimes_{\mathbb{F}}\Xi_{\boldsymbol{W}}.
\]
\end{defn}

Proposition~\ref{prop:sigma-cheq} is the key algebraic statement;
we show this in Appendix~\ref{sec:collapsing-colors-proof}.
\begin{prop}
\label{prop:sigma-cheq}Consider the $R_{\boldsymbol{X}}$-action
$\rho_{\boldsymbol{W}}\circ R_{\sigma}$ on $\overline{C}_{\boldsymbol{W}}$.
The map $\overline{\sigma}:\overline{C}_{\boldsymbol{X}}\to\overline{C}_{\boldsymbol{W}}$
is an $R_{\boldsymbol{X}}$-chain homotopy equivalence.
\end{prop}

\begin{defn}
Define the \emph{``color collapsing''} forgetful functor $\mathsf{Tang}_{\sigma}^{\mathrm{gen}}:\mathsf{Tang}_{\boldsymbol{e},\boldsymbol{W}}^{\mathrm{gen}}\to\mathsf{Tang}_{\boldsymbol{e},\boldsymbol{X}}^{\mathrm{gen}}$
as follows. The image of an object $(T,\vec{\boldsymbol{p}})\in\mathsf{Tang}_{\boldsymbol{e},\boldsymbol{W}}^{\mathrm{gen}}$
is $(T,\vec{\boldsymbol{q}})\in\mathsf{Tang}_{\boldsymbol{e},\boldsymbol{X}}^{\mathrm{gen}}$
where $\boldsymbol{q}_{x}:=\bigsqcup_{\sigma(w)=x}\boldsymbol{p}_{w}$.
Similarly, the image of a morphism $(\Sigma,\vec{A})$ is $(\Sigma,\vec{B})\in\mathrm{Mor}(\mathsf{Tang}_{\boldsymbol{e},\boldsymbol{X}}^{\mathrm{gen}})$
where $B_{x}:=\bigsqcup_{\sigma(w)=x}A_{w}$.
\end{defn}

\begin{prop}[Collapsing colors]
\label{prop:collapsing-colors}Let $(\Sigma,\vec{A}):(T,\vec{\boldsymbol{p}})\to(T',\vec{\boldsymbol{p}'})$
be a morphism in $\mathsf{Tang}_{\boldsymbol{e},\boldsymbol{W}}^{\mathrm{gen}}$,
and denote $\mathsf{Tang}_{\sigma}^{\mathrm{gen}}(\Sigma,\vec{A})$
as $(\Sigma,\vec{B}):(T,\vec{\boldsymbol{q}})\to(T',\vec{\boldsymbol{q}'})$.
The following square commutes.
\[\begin{tikzcd}[ampersand replacement=\&]
	{\overline{D}_{\mathsf{BN}}(T,\vec{\boldsymbol{q}})} \& {\overline{D}_{\mathsf{BN}}(T,\vec{\boldsymbol{p}})} \\
	{\overline{D}_{\mathsf{BN}} (T',\vec{\boldsymbol{q}'})} \& {\overline{D}_{\mathsf{BN}}(T',\vec{\boldsymbol{p}'})}
	\arrow["{\overline{\sigma}}", from=1-1, to=1-2]
	\arrow["{\overline{D}_{\mathsf{BN}}(\Sigma ,\vec{B})}"{description}, from=1-1, to=2-1]
	\arrow["{\overline{D}_{\mathsf{BN}}(\Sigma ,\vec{A})}"{description}, from=1-2, to=2-2]
	\arrow["{\overline{\sigma}}", from=2-1, to=2-2]
\end{tikzcd}\]
\end{prop}

\begin{proof}
Routine check for each elementary movie. It is clear for (\ref{enu:movie1})-(\ref{enu:movie6}).
For the slide movie, let the relevant basepoints have color $w\in\boldsymbol{W}$.
The commutativity of the square follows from $\xi_{w}\circ\Xi_{\sigma}=\Xi_{\sigma}\circ\xi_{\sigma(w)}$.
\end{proof}
\begin{proof}[Proof of Proposition~\ref{prop:check-movie-moves-local}]
Let $\boldsymbol{W}:=\bigsqcup_{i=1}^{d}\boldsymbol{X}_{i}$ and
let $\sigma:\boldsymbol{W}\to\boldsymbol{X}$ be the obvious surjection.
Let $(\Sigma,\vec{A}_{\boldsymbol{W}}),(\Sigma',\vec{A'}_{\boldsymbol{W}})\in\mathrm{Mor}(\mathsf{Tang}_{\boldsymbol{e},\boldsymbol{W}}^{\mathrm{gen}})$
be the composition of the $(\Sigma_{i},\vec{A_{i}})$'s and $(\Sigma_{i}',\vec{A_{i}'})$'s
where we view the $\boldsymbol{X}_{i}$'s as pairwise disjoint subsets
of $\boldsymbol{W}$. Then $(\Sigma,\vec{A})=\mathsf{Tang}_{\sigma}^{\mathrm{gen}}(\Sigma,\vec{A}_{\boldsymbol{W}})$
and $(\Sigma',\vec{A'})=\mathsf{Tang}_{\sigma}^{\mathrm{gen}}(\Sigma',\vec{A'}_{\boldsymbol{W}})$,
and by Proposition~\ref{prop:composing-tangles-cobordisms} $\overline{D}_{\mathsf{BN}}(\Sigma,\vec{A}_{\boldsymbol{W}})$
and $\overline{D}_{\mathsf{BN}}(\Sigma',\vec{A'}_{\boldsymbol{W}})$
are $R_{\boldsymbol{W}}$-chain homotopic. By Proposition~\ref{prop:collapsing-colors}
\[
\overline{\sigma}\circ\overline{D}_{\mathsf{BN}}(\Sigma,\vec{A})=\overline{D}_{\mathsf{BN}}(\Sigma,\vec{A}_{\boldsymbol{W}})\circ\overline{\sigma}\sim\overline{D}_{\mathsf{BN}}(\Sigma',\vec{A'}_{\boldsymbol{W}})\circ\overline{\sigma}=\overline{\sigma}\circ\overline{D}_{\mathsf{BN}}(\Sigma',\vec{A'})
\]
are $R_{\boldsymbol{X}}$-chain homotopic. Proposition~\ref{prop:check-movie-moves-local}
follows since $\overline{\sigma}$ is an $R_{\boldsymbol{X}}$-chain
homotopy equivalence by Proposition~\ref{prop:sigma-cheq}.
\end{proof}

\subsection{$\overline{D}_{\mathsf{BN}}$-simple tangles}

\begin{figure}[h]
\begin{centering}
\includegraphics[scale=3.3]{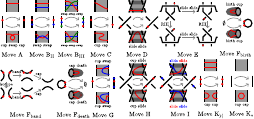}
\par\end{centering}
\caption{\label{fig:H23R3}A non-exhaustive list of decorated movie moves}
\end{figure}
In Subsection~\ref{subsec:Invariance-under-changing}, we will check
that $\overline{D}_{\mathsf{BN}}$ is invariant under certain moves,
some of which are drawn in Figure~\ref{fig:H23R3}. Move E of Figure~\ref{fig:H23R3},
i.e.\ that Reidemeister III and slide commute, turns out to be particularly
complicated to check by hand (although it is not impossible).\footnote{In fact, this is the only move that is complicated to check; all the
other moves are straightforward.} Instead of checking all the moves by hand, we adapt Bar-Natan's argument
\cite[Section~8.3]{BN1} (compare \cite{MR2171235}) to our setting.
\begin{defn}[{\cite[Lemma~8.5]{BN1}}]
An $\boldsymbol{X}$-pointed tangle $(T,\vec{\boldsymbol{p}})\in\mathsf{Tang}_{\boldsymbol{e},\boldsymbol{X}}^{\mathrm{gen}}$
is \emph{$\overline{D}_{\mathsf{BN}}$-simple} if every bidegree $(0,0)$
$R_{\boldsymbol{X}}$-linear homotopy equivalence of $\overline{D}_{\mathsf{BN}}(T,\vec{\boldsymbol{p}})$
with itself is $R_{\boldsymbol{X}}$-homotopic to the identity.
\end{defn}

\begin{lem}
Any $\boldsymbol{X}$-pointed tangle $(T,\vec{\boldsymbol{p}})\in\mathsf{Tang}_{\boldsymbol{e},\boldsymbol{X}}^{\mathrm{gen}}$
such that $T$ is a pairing (\cite[Lemma~8.6]{BN1}) is $\overline{D}_{\mathsf{BN}}$-simple.
In fact, the identity is the unique nonzero, bidegree $(0,0)$ $R_{\boldsymbol{X}}$-linear
endomorphism of $\overline{D}_{\mathsf{BN}}(T,\vec{\boldsymbol{p}})=C_{\mathsf{BN}}(T,\vec{\boldsymbol{p}})\otimes_{\mathbb{F}}(R_{\boldsymbol{X}}\otimes_{\mathbb{F}}\Xi_{\boldsymbol{X}})$.
\end{lem}

\begin{proof}
First, \cite[Lemma~8.6]{BN1} shows that the identity is the unique
nonzero, bidegree $(0,0)$ endomorphism of $C_{\mathsf{BN}}(T)$.
Also, a similar argument shows that $\mathrm{End}_{\mathsf{BN}_{\boldsymbol{e}}}(C_{\mathsf{BN}}(T))$
is supported in bidegree $(0,q)$ for $q\le0$. Hence, the identity
is the unique nonzero, bidegree $(0,0)$ $R_{\boldsymbol{X}}$-linear
endomorphism of $C_{\mathsf{BN}}(T)\otimes_{\mathbb{F}}R_{\boldsymbol{X}}$.

Let $\varphi$ be a nonzero, bidegree $(0,0)$ $R_{\boldsymbol{X}}$-linear
endomorphism of $\overline{D}_{\mathsf{BN}}(T,\vec{\boldsymbol{p}})$.
For $n\ge0$, let $V_{n}$ be the $\mathbb{F}$-vector subspace of
$\Xi_{\boldsymbol{X}}$ spanned by $\{\prod_{x\in\boldsymbol{X}}\xi_{x}^{-n_{x}}:\sum_{x\in\boldsymbol{X}}n_{x}=n\}$.
Since $\overline{D}_{\mathsf{BN}}(T,\vec{\boldsymbol{p}})=\bigoplus_{n\ge0}C_{\mathsf{BN}}(T)\otimes_{\mathbb{F}}R_{\boldsymbol{X}}\otimes_{\mathbb{F}}V_{n}$,
we have that $\varphi$ consists of maps ${\rm Id}_{C_{\mathsf{BN}}(T)}\otimes{\rm Id}_{R_{\boldsymbol{X}}}\otimes f_{n}$
for some $\mathbb{F}$-linear maps $f_{n}:V_{n}\to V_{n}$ of bidegree
$(0,0)$, by the previous paragraph. Now, $\partial\varphi=\varphi\partial$
in particular gives 
\[
\sum_{x}{\rm Id}_{C_{\mathsf{BN}}(T)}\otimes x\otimes(\xi_{x}f_{n})=\sum_{x}{\rm Id}_{C_{\mathsf{BN}}(T)}\otimes x\otimes(f_{n-1}\xi_{x}),
\]
and so $\xi_{x}f_{n}=f_{n-1}\xi_{x}$ for all $x,n$.

We claim that if $f_{0}=0$, then $f_{n}=0$ for all $n$, and if
$f_{0}=\mathrm{Id}$, then $f_{n}=\mathrm{Id}$ for all $n$. If $f_{0}=\mathrm{Id}$,
then replace $f_{n}$ by $f_{n}-\mathrm{Id}_{V_{n}}$; hence we may
assume $f_{0}=0$. Let us induct on $n$. Let $n\ge1$ and assume
that $f_{n-1}=0$; then $\xi_{x}f_{n}=0$ for all $x$. Now, $\bigcap_{x\in\boldsymbol{X}}\mathrm{ker}(\xi_{x})$
is disjoint from $V_{n}$ since it is the dimension $1$ subspace
spanned by $1\in\Xi_{\boldsymbol{X}}$. Hence $f_{n}=0$ and the claim
follows. Thus if $\varphi$ is nonzero, then it is the identity, i.e.\ the
lemma follows.
\end{proof}
\begin{prop}[{\cite[Lemmas~8.7-8.9]{BN1}}]
\label{prop:dkh-simple}If $(T,\vec{\boldsymbol{p}})$ is $\overline{D}_{\mathsf{BN}}$-simple,
then any tangle obtained from it by adding a crossing along the boundary
is $\overline{D}_{\mathsf{BN}}$-simple. If there exists some $(\Sigma,\vec{A}):(T,\vec{\boldsymbol{p}})\to(T',\vec{\boldsymbol{p}}')$
such that $\Sigma$ is induced by an isotopy $T\to T'$, then $(T',\vec{\boldsymbol{p}'})$
is also $\overline{D}_{\mathsf{BN}}$-simple. Moreover, any two bidegree
$(0,0)$ $R_{\boldsymbol{X}}$-homotopy equivalences $f,g:\overline{D}_{\mathsf{BN}}(T,\vec{\boldsymbol{p}})\to\overline{D}_{\mathsf{BN}}(T',\vec{\boldsymbol{p}'})$
are $R_{\boldsymbol{X}}$-homotopic.
\end{prop}

\begin{proof}
This follows from the same arguments as \cite[Section~8.3]{BN1},
using Lemma~\ref{lem:homotopy-equivalence}. E.g.\ for the last
sentence: let $\varphi:\overline{D}_{\mathsf{BN}}(T,\vec{\boldsymbol{p}})\to\overline{D}_{\mathsf{BN}}(T',\vec{\boldsymbol{p}'})$
and $\psi:\overline{D}_{\mathsf{BN}}(T',\vec{\boldsymbol{p}'})\to\overline{D}_{\mathsf{BN}}(T,\vec{\boldsymbol{p}})$
be bidegree $(0,0)$ $R_{\boldsymbol{X}}$-homotopy inverses. Since
$(T,\vec{\boldsymbol{p}})$ is $\overline{D}_{\mathsf{BN}}$-simple,
$\psi\circ f,\psi\circ g\sim{\rm Id}_{\overline{D}_{\mathsf{BN}}(T,\vec{\boldsymbol{p}})}$,
and so they are homotopic. Hence, $f\sim\varphi\circ\psi\circ f\sim\varphi\circ\psi\circ g\sim g$.
\end{proof}

\subsection{\label{subsec:Invariance-under-changing}Invariance under changing
$\vec{A}$}

We first study which moves need to be checked to show invariance under
homotoping $\vec{A}$, which may be of independent interest. Then,
we will show that $\overline{D}_{\mathsf{BN}}(\Sigma,\vec{A})$ for
fixed $\Sigma$ only depends on the homology class of $\vec{A}$ by
checking two extra moves.

It will be convenient to split \ref{enu:conde-j} into six conditions
\ref{enu:conde}-\ref{enu:condj}. \ref{enu:conde}-\ref{enu:condh}
ensure that the Reidemeister, Morse, cap, cup, swap, and slide loci
are pairwise distinct; \ref{enu:condi} and \ref{enu:condj} ensure
that they are supported in pairwise distinct times.
\begin{enumerate}
[label=(Cond~\Alph*)]\addtocounter{enumi}{4}
\item \label{enu:conde}$A$ is disjoint from the Reidemeister loci.
\item $A$ is disjoint from the Morse loci.
\item The cap, cup, and slide loci are not double points.
\item \label{enu:condh}$Cr$ does not intersect the cap, cup, swap loci.
\item \label{enu:condi}For each crossing $c$, the slide loci for $c$
are supported in distinct times.
\item \label{enu:condj}Distant Reidemeister, Morse, cap, cup, swap, and
slide loci are supported in distinct times.
\end{enumerate}

For each of \ref{enu:conda}-\ref{enu:condj}, failing it is a codimension
$2$ condition. Hence, given a smooth homotopy $(\varphi_{t}):A\to\Sigma$
between two smooth maps $\varphi_{0},\varphi_{1}:A\to\Sigma$ that
satisfy \ref{enu:conda}-\ref{enu:condj}, we may homotope $(\varphi_{t})$
rel.\ $\partial$ such that for each of \ref{enu:conda}-\ref{enu:condj},
$(\varphi_{t})$ fails it only at a finite number of points. Examining
$\varphi_{t}$ for $t$ that are near these points, we see that we
are left to show that $\overline{D}_{\mathsf{BN}}(\Sigma_{1},\vec{A}_{1})$,
$\overline{D}_{\mathsf{BN}}(\Sigma_{2},\vec{A}_{2})$ are the same
for $(\Sigma_{1},\vec{A_{1}})$ and $(\Sigma_{2},\vec{A_{2}})$ that
differ in the following ways: (some (non-exhaustive) examples are
given in Figure~\ref{fig:H23R3}.)
\begin{enumerate}
[label=(Move~\Alph*)]
\item Reidemeister I type moves (Move A of Figure~\ref{fig:H23R3})
\item Reidemeister II and III type moves (Moves $\mathrm{B_{II}}$ and $\mathrm{B_{III}}$
of Figure~\ref{fig:H23R3})
\item Creating or annihilating a cap and cup locus 
\item Creating or annihilating two cancelling slide moves
\item An arc goes across a Reidemeister locus
\item \label{enu:move-rm-strand}An arc goes across a Morse locus
\item An arc goes across a cap, cup, swap, or slide locus
\item A cap, cup, or swap locus goes across $Cr$
\item Swapping the order of two slide movies for the same crossing that
are adjacent in time
\item \label{enu:move-far-commutation}Swapping the order of two distant
Reidemeister, Morse, cap, cup, swap, or slide movies that are adjacent
in time
\end{enumerate}
Finally, to check that $\overline{D}_{\mathsf{BN}}(\Sigma,\vec{A})$
only depends on the mod $2$ homology class of $\vec{A}$, we need
to check the following:
\begin{enumerate}
[resume, label=(Move~\Alph*)]
\item \label{enu:move-homology}Moves $\mathrm{K}_{)(}$ and $\mathrm{K}_{\circ}$
of Figure~\ref{fig:H23R3}
\end{enumerate}
All the moves except \ref{enu:move-rm-strand} and \ref{enu:move-far-commutation}
are supported in a region whose underlying surface is given by an
undecorated Reidemeister movie, or $I\times T$ where $T$ is either
a $2$-ended tangle without any crossings or a $4$-ended tangle with
one crossing. Hence, by Lemma~\ref{lem:homotopy-equivalence}, for
all these cases, the corresponding cobordism map $\overline{D}_{\mathsf{BN}}$
is a homotopy equivalence, and so $\overline{D}_{\mathsf{BN}}$ is
invariant under these moves by Proposition~\ref{prop:dkh-simple}.

\ref{enu:move-far-commutation} and \ref{enu:move-rm-strand} are
easy to check. \ref{enu:move-far-commutation} strictly commutes on
the chain level by Proposition~\ref{prop:composing-tangles-cobordisms}.
Assuming the moves we have checked, \ref{enu:move-rm-strand} for
Morse loci reduces to Moves $\mathrm{F_{birth}}$, $\mathrm{F_{band}}$,
and $\mathrm{F_{death}}$ of Figure~\ref{fig:H23R3}, which strictly
commute on the chain level.

\subsection{\label{subsec:The-general-case}The general case}

In this subsection we complete the proof of \ref{enu:step3} for $\overline{D}_{\mathsf{BN}}$.
By Carter and Saito \cite{CS} and \ref{enu:step2} for $\overline{D}_{\mathsf{BN}}$,
we reduce to showing $\overline{D}_{\mathsf{BN}}(\Sigma_{1},\vec{A_{1}})\sim\overline{D}_{\mathsf{BN}}(\Sigma_{2},\vec{A_{2}})$
for generic $\boldsymbol{X}$-decorated tangle cobordisms $(\Sigma_{1},\vec{A_{1}})$
and $(\Sigma_{2},\vec{A_{2}})$ such that they represent the same
morphism in $\mathsf{Tang}_{\boldsymbol{e},\boldsymbol{X}}$, and
$\Sigma_{1}$ and $\Sigma_{2}$ are compositions of undecorated Reidemeister
and Morse movies that are related by a single far commutation or movie
move. Assume that these moves occur in $[a,b]\times D$ for some region
$D\subset D^{3}$.

We say that $\vec{A_{1}}$ \emph{does not interact with the far commutation
or movie move} if $\vec{A_{1}}$ does not intersect with $[a,b]\times D$
and is vertical between the Reidemeister and Morse movies that are
involved in the far commutation or movie move.
\begin{lem}
\label{lem:no-interaction}If $\vec{A_{1}}$ does not interact with
the far commutation or movie move, and if $(\Sigma_{2},\vec{A}_{2}')$
is the decorated cobordism obtained by performing the far commutation
or movie move, then $\overline{D}_{\mathsf{BN}}(\Sigma_{1},\vec{A_{1}})$
and $\overline{D}_{\mathsf{BN}}(\Sigma_{2},\vec{A_{2}}')$ are $R_{\boldsymbol{X}}$-homotopic.
\end{lem}

\begin{proof}
The decorated cobordisms $(\Sigma_{1},\vec{A_{1}})$ and $(\Sigma_{2},\vec{A_{2}'})$
agree outside of $[a,b]\times D^{3}$. Hence, we may restrict to $[a,b]\times D^{3}$,
in which case 
\[
\overline{D}_{\mathsf{BN}}(\Sigma_{1},\vec{A_{1}})=C_{\mathsf{BN}}(\Sigma_{1})\otimes\mathrm{Id}_{R_{\boldsymbol{X}}\otimes\Xi_{\boldsymbol{X}}}\ \mathrm{and}\ \overline{D}_{\mathsf{BN}}(\Sigma_{2},\vec{A_{2}'})=C_{\mathsf{BN}}(\Sigma_{2})\otimes\mathrm{Id}_{R_{\boldsymbol{X}}\otimes\Xi_{\boldsymbol{X}}},
\]
and so they are $R_{\boldsymbol{X}}$-homotopic since $C_{\mathsf{BN}}(\Sigma_{1})$
and $C_{\mathsf{BN}}(\Sigma_{2})$ are homotopic.
\end{proof}
Now, the general case follows from the following lemma.
\begin{lem}
\label{lem:isotope-as}Given any $\vec{A}_{1}$, there exists an $\vec{A_{1}'}$
that is isotopic rel.\ $\partial$ to $\vec{A_{1}}$ such that $\vec{A_{1}'}$
does not interact with the far commutation or movie move.
\end{lem}

\begin{proof}
We first claim that if $D\subset D^{3}$ is some region, $\Sigma$
is vertical in $[a,b]\times(D^{3}\setminus D)$, and $\vec{A}$ is
disjoint from $[a,b]\times D$, then we can isotope $\vec{A}$ in
$\Sigma$ rel.\ $[a,b]\times D$ such that it becomes vertical in
$[a,b]\times D^{3}$. Indeed, this can be achieved by first applying
a small isotopy to $\vec{A}$ to make it vertical in $[a,a+\varepsilon]\times(D^{3}\setminus D)$,
and then applying an ambient isotopy of $\Sigma$ that expands $[a,a+\varepsilon]\times(D^{3}\setminus D)$
to $[a,b]\times(D^{3}\setminus D)$.

Let us show the lemma. Let us first consider the case where $\Sigma_{1}$
corresponds to the far commutation move. Say the two Reidemeister
or Morse movies happen in $[t_{1}-\varepsilon,t_{1}+\varepsilon]\times D_{1}$
and $[t_{2}-\varepsilon,t_{2}+\varepsilon]\times D_{2}$, respectively,
where $t_{1}<t_{2}$ and $D_{1},D_{2}\subset D^{3}$ are disjoint.
We may assume that $\vec{A_{1}}$ is disjoint from $[t_{i}-2\varepsilon,t_{i}+2\varepsilon]\times D_{i}$
and is vertical in $[t_{i}-2\varepsilon,t_{i}+2\varepsilon]\times(D^{3}\setminus D_{i})$.
Now, apply an ambient isotopy of $\Sigma_{1}$ that expands $[t_{1}+\varepsilon,t_{1}+2\varepsilon]\times D_{1}$
to $[t_{1}+\varepsilon,t_{2}+2\varepsilon]\times D_{1}$ and expands
$[t_{2}-2\varepsilon,t_{2}-\varepsilon]\times D_{2}$ to $[t_{1}-2\varepsilon,t_{2}-\varepsilon]\times D_{2}$.
Then we can apply the above claim for $D=D_{1}\sqcup D_{2}$ and $[a,b]=[t_{1}-2\varepsilon,t_{2}+2\varepsilon]$.

Let us consider the case where $\Sigma_{1}$ corresponds to a movie
move, say supported in $[a,b]\times D$. Then since the ambient surfaces
of each of the movie moves are disjoint unions of disks, we can isotope
$\vec{A}$ so that it is disjoint from $[a,b]\times D$, and we can
then apply the above claim.
\end{proof}
\begin{proof}[Proof of \ref{enu:step3} for $\overline{D}_{\mathsf{BN}}$]
Let $\vec{A_{1}'}$ be as in Lemma~\ref{lem:isotope-as}, and let
$(\Sigma_{2},\vec{A}_{2}')$ be the decorated cobordism obtained by
performing the far commutation or movie move. Then $\overline{D}_{\mathsf{BN}}(\Sigma_{1},\vec{A_{1}})\sim\overline{D}_{\mathsf{BN}}(\Sigma_{1},\vec{A_{1}'})\sim\overline{D}_{\mathsf{BN}}(\Sigma_{2},\vec{A_{2}'})\sim\overline{D}_{\mathsf{BN}}(\Sigma_{2},\vec{A_{2}})$,
where the first and third follow from Subsection~\ref{subsec:Invariance-under-changing}
and the second follows from Lemma~\ref{lem:no-interaction}.
\end{proof}

\subsection{\label{subsec:The-sweep-around-move}The sweep-around move}

\begin{figure}[h]
\begin{centering}
\includegraphics[scale=3.5]{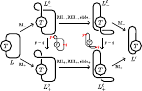}
\par\end{centering}
\caption{\label{fig:sweeparound}A sweep-around move}
\end{figure}

In this subsection we carry out \ref{enu:step4} for $\overline{D}_{\mathsf{BN}}$.
By Subsection~\ref{subsec:Invariance-under-changing}, it reduces
to checking that $\overline{D}_{\mathsf{BN}}$ of Figure~\ref{fig:sweeparound}
commutes for all $\boldsymbol{X}$-pointed tangle $T$ with two endpoints.

Our proof is a straightforward adaptation of Morrison, Walker, and
Wedrich's proof \cite{MWW} to our setting. Let us briefly recall
their proof, following the notations of \cite[Proposition~3.7]{MR4504654},\footnote{We note a mistake in the proof of \cite[Proposition~3.7]{MR4504654}:
the leftmost and rightmost triangles in the commutative diagram of
\cite[page~8818]{MR4504654} do not commute. The vertical maps should
be as in Diagram~(\ref{eq:sweep-around-comdiag}), and an additional
argument like \cite[Lemma~3.9]{MWW} is necessary (this corresponds
to that the top squares of Diagram~(\ref{eq:sweep-around-comdiag})
commute).}\footnote{We also note that \cite[page~8817]{MR4504654} has an unneeded hypothesis:
it is unnecessary to start with a braid closure. In particular, the
claim ${\cal C}_{0}^{-}(L_{+}^{i})\cong{\cal C}_{0}^{-}(L_{-}^{i})$
holds even when one starts with an arbitrary tangle $T$ as in Figure~\ref{fig:sweeparound}.} where a concise proof for $C_{\mathsf{BN}}$ is written. Let $T$
be an arbitrary tangle with two endpoints, and consider Figure~\ref{fig:sweeparound}.
The horizontal arrows are the composition of Reidemeister II and III
movies $\Sigma_{\pm}^{i}:L_{\pm}^{i}\to L_{\pm}^{i+1}$. They show
that the two maps $C_{\mathsf{BN}}(L)\to C_{\mathsf{BN}}(L')$ given
by the composition of the top maps and the bottom maps of Figure~\ref{fig:sweeparound}
are homotopic. To show this, rather than working merely up to chain
homotopy, they choose specific chain level representatives for the
undecorated Reidemeister chain maps to prove that the two maps $C_{\mathsf{BN}}(L)\to C_{\mathsf{BN}}(L')$
agree on the chain level; we use the same choices for the Reidemeister
maps.

For each of $L,L_{\pm}^{i},L'$, define the \emph{external grading}
of $C_{\mathsf{BN}}$ as the contribution to the homological grading
given by the crossings that involve the horizontal arc that moves.
Denote as $C_{\mathsf{BN},0}$ the summand with external grading $0$.
Then, $C_{\mathsf{BN}}(L)$ and $C_{\mathsf{BN}}(L')$ are supported
in external grading $0$, and for each $i$, there exists a natural
isomorphism $C_{\mathsf{BN},0}(L_{-}^{i})\cong C_{\mathsf{BN},0}(L_{+}^{i})$.
Now, \cite{MWW} reduces checking the sweep-around move to local checks:
they check that every undecorated Reidemeister move preserves or decreases
the external grading, and that every quadrilateral of Diagram~(\ref{eq:sweep-around-comdiag})
strictly commutes, where the horizontal maps are the $C_{\mathsf{BN},0}(L_{\pm}^{i})\to C_{\mathsf{BN},0}(L_{\pm}^{i+1})$
components of the Reidemeister maps. 
\begin{equation}\label{eq:sweep-around-comdiag}
\begin{tikzcd}[ampersand replacement=\&,column sep=tiny,row sep=small]
	\& {C_{\mathsf{BN},0}(L^0_-)} \& {\cdots } \& {C_{\mathsf{BN},0}(L^i_-)} \& \cdots \& {C_{\mathsf{BN},0}(L^\ell_-)} \& \\
	{C_{\mathsf{BN},0}(L)} \& {C_{\mathsf{BN},0}(L^0_-)} \& \cdots \& {C_{\mathsf{BN},0}(L^i_-)} \& \cdots \& {C_{\mathsf{BN},0}(L^\ell_-)} \& {C_{\mathsf{BN},0}(L')} \\
	\& {C_{\mathsf{BN},0}(L^0_+)} \& \cdots \& {C_{\mathsf{BN},0}(L^i_+)} \& \cdots \& {C_{\mathsf{BN},0}(L^\ell_+)}
	\arrow[from=1-2, to=1-3]
	\arrow["{p-q}", from=1-2, to=2-2]
	\arrow[from=1-3, to=1-4]
	\arrow[from=1-4, to=1-5]
	\arrow["{p-q}", from=1-4, to=2-4]
	\arrow[from=1-5, to=1-6]
	\arrow["{p-q}", from=1-6, to=2-6]
	\arrow["{\mathrm{RI}_-}", from=1-6, to=2-7]
	\arrow["{\mathrm{RI}_-}", from=2-1, to=1-2]
	\arrow["{\mathrm{RI}_+}"', from=2-1, to=3-2]
	\arrow[from=2-2, to=2-3]
	\arrow["\cong", from=2-2, to=3-2]
	\arrow[from=2-3, to=2-4]
	\arrow[from=2-4, to=2-5]
	\arrow["\cong", from=2-4, to=3-4]
	\arrow[from=2-5, to=2-6]
	\arrow["\cong", from=2-6, to=3-6]
	\arrow[from=3-2, to=3-3]
	\arrow[from=3-3, to=3-4]
	\arrow[from=3-4, to=3-5]
	\arrow[from=3-5, to=3-6]
	\arrow["{\mathrm{RI}_+}"', from=3-6, to=2-7]
\end{tikzcd}
\end{equation}

Now, let us consider the case where $T$ is an $\boldsymbol{X}$-pointed
tangle. The horizontal arrows of Figure~\ref{fig:sweeparound} are
compositions of Reidemeister II, Reidemeister III, and slide movies
$(\Sigma_{\pm}^{i},\vec{A_{\pm}^{i}}):(L_{\pm}^{i},\vec{\boldsymbol{p}_{\pm}^{i}})\to(L_{\pm}^{i+1},\vec{\boldsymbol{p}_{\pm}^{i+1}})$.
To run the same argument, we also need to choose specific chain level
representatives for the decorated link cobordism map $\overline{D}_{\mathsf{BN}}$
for Reidemeister and slide movies. If $(\Sigma,\vec{A})$ is a Reidemeister
movie, let the chain level representative be $C_{\mathsf{BN}}(\Sigma)\otimes\mathrm{Id}_{R_{\boldsymbol{X}}\otimes\Xi_{\boldsymbol{X}}}$
where $C_{\mathsf{BN}}(\Sigma)$ is the chain level representative
of \cite{MWW}. If $(\Sigma_{\pm}^{i},\vec{A_{\pm}^{i}})$ is a slide
movie, recall that the underlying $\mathbb{F}$-modules of $\overline{D}_{\mathsf{BN}}(L_{\pm}^{i})$
and $\overline{D}_{\mathsf{BN}}(L_{\pm}^{i+1})$ are equal. Choose
the following chain level representative (which is the same as Definition~\ref{def:dbarbn}):
\begin{equation}
\overline{D}_{\mathsf{BN}}(\Sigma_{\pm}^{i},\vec{A_{\pm}^{i}})=\mathrm{Id}_{C_{\mathsf{BN}}(L_{\pm}^{i})}\otimes\mathrm{Id}_{R_{\boldsymbol{X}}\otimes\Xi_{\boldsymbol{X}}}+H\otimes\mathrm{Id}_{R_{\boldsymbol{X}}}\otimes\xi_{y}\label{eq:slide-map}
\end{equation}
where $H$ is the basepoint sliding homotopy for the corresponding
crossing.

Define the external grading of $\overline{D}_{\mathsf{BN}}(L_{\pm}^{i},\vec{\boldsymbol{p}_{\pm}^{i}})=C_{\mathsf{BN}}(L_{\pm}^{i})\otimes R_{\boldsymbol{X}}\otimes\Xi_{\boldsymbol{X}}$
by declaring $R_{\boldsymbol{X}}\otimes\Xi_{\boldsymbol{X}}$ to be
supported in external grading $0$. Reidemeister maps on $\overline{D}_{\mathsf{BN}}$
preserve or decrease the external grading by the corresponding statement
for $C_{\mathsf{BN}}$. Slide maps also preserve or decrease the external
grading: the first summand of Equation~(\ref{eq:slide-map}) preserves
the external grading, and the second summand decreases the external
grading by $1$.

Consider Diagram~(\ref{eq:sweep-around-comdiag}), but where all
the $C_{\mathsf{BN},0}$'s are replaced by $\overline{D}_{\mathsf{BN},0}$.
We claim that every quadrilateral strictly commutes. That the leftmost
and rightmost quadrilateral, and all the squares that involve Reidemeister
II or III moves strictly commute follows from the corresponding statements
for $C_{\mathsf{BN}}$. The squares that involve slide moves commute
since the horizontal maps are $\mathrm{Id}_{\overline{D}_{\mathsf{BN},0}}$:
indeed, the component of Equation~(\ref{eq:slide-map}) that preserves
the external grading is $\mathrm{Id}_{\overline{D}_{\mathsf{BN}}}$.
This completes the proof of the sweep-around move, and hence \ref{enu:step4},
for $\overline{D}_{\mathsf{BN}}$.
\begin{proof}[Proof of Theorem~\ref{thm:enhanced-khovanov}]
Let us work in the setting of Section~\ref{sec:Pointed-links-and},
which is more general than Subsection~\ref{subsec:An-enhancement-of}.
(\ref{enu:141}) is Definition~\ref{def:dkh-objects}, and this section
proves (\ref{enu:142}). Let us show (\ref{enu:144}). Let $(\Sigma,\vec{A}):(L,\vec{\boldsymbol{p}})\to(L',\vec{\boldsymbol{p}'})$
be an $\boldsymbol{X}$-decorated link cobordism, let $f_{DKh(L,\vec{\boldsymbol{p}})}:DKh(L,\vec{\boldsymbol{p}})\to\overline{D}Kh(L,\vec{\boldsymbol{p}})$
be the $\mathbb{F}$-linear map from Remark~\ref{rem:f-linear-inverse},
and let $q_{DKh(L',\vec{\boldsymbol{p}'})}:\overline{D}Kh(L',\vec{\boldsymbol{p}'})\to DKh(L',\vec{\boldsymbol{p}'})$
be our preferred $R_{\boldsymbol{X}}$-quasi-isomorphism from Definition~\ref{def:preferred}.
Then, the image of $DKh(\Sigma,\vec{A})$ under the forgetful functor
$D^{b}(\mathsf{Mod}_{R_{\boldsymbol{X}}})\to K^{b}(\mathsf{Mod}_{\mathbb{F}})$
is $q_{DKh(L',\vec{\boldsymbol{p}'})}\circ\overline{D}Kh(\Sigma,\vec{A})\circ f_{DKh(L,\vec{\boldsymbol{p}})}$,
and it is straightforward to check that this agrees with $CKh(\Sigma)$
for all the elementary movies $(\Sigma,\vec{A})$.

In Appendix~\ref{sec:Comparison-with-previous} we relate our construction
with the constructions of Baldwin-Levine-Sarkar and Lipshitz-Sarkar
and prove (\ref{enu:For-the-decorated}) for these cases. Note that
Hedden-Ni considers the $R_{\boldsymbol{X}}$-action on \emph{homology},
and so the statement for Hedden-Ni is equivalent to saying that $DKh$
and $CKh$ agree on homology, which is (\ref{enu:144}).
\end{proof}

\section{\label{sec:Standard-'s-with}The standard $\mathbb{RP}^{2}$ with
Euler number $-2$}

In this section, we study the effect of taking the connected sum with
the standard $\mathbb{RP}^{2}$ with Euler number $-2$. Then, we
define a specialized version of the decorated Khovanov TQFT that behaves
particularly nicely under this operation.
\begin{lem}
\label{lem:model-rp2-computation}Let $\boldsymbol{X}:=\{x\}$ and
let $((I\times T)\#\mathbb{RP}^{2},\mathbb{RP}^{1}):(T,\emptyset)\to(T,\emptyset)$
be the $\boldsymbol{X}$-decorated cobordism of Figure~\ref{fig:rp2-tangle},
where the decoration has color $x$. Then,
\[
\overline{D}_{\mathsf{BN}}((I\times T)\#\mathbb{RP}^{2},\mathbb{RP}^{1}):C_{\mathsf{BN}}(T)\otimes_{\mathbb{F}}R_{x}\otimes_{\mathbb{F}}\Xi_{x}\to C_{\mathsf{BN}}(T)\otimes_{\mathbb{F}}R_{x}\otimes_{\mathbb{F}}\Xi_{x}
\]
is multiplication by $\xi_{x}$.
\end{lem}

\begin{proof}
Direct computation. See Figure~\ref{fig:rp2-tangle}: the dashed
arrows compose to the identity.
\end{proof}
\begin{figure}[h]
\begin{centering}
\includegraphics[scale=2.2]{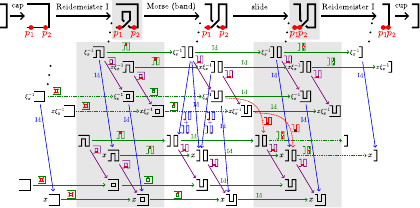}
\par\end{centering}
\caption{\label{fig:rp2-tangle}The decorated cobordism map for taking the
connected sum with a standard $\mathbb{RP}^{2}$ with Euler number
$-2$. Only the subcomplexes $C_{\mathsf{BN}}\otimes_{\mathbb{F}}R_{x}\otimes_{\mathbb{F}}(\xi_{x}^{-1}\mathbb{F}\oplus\mathbb{F})$
are drawn, and the grading shifts are omitted.}
\end{figure}

\begin{figure}[h]
\begin{centering}
\includegraphics[scale=2.2]{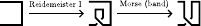}
\par\end{centering}
\caption{\label{fig:rp2-tangle-2}A movie that corresponds to the standard
$\mathbb{RP}^{2}$ with Euler number $2$.}
\end{figure}

\begin{rem}[Standard $\mathbb{RP}^{2}$ with Euler number $2$]
Note that $C_{\mathsf{BN}}$ for the decorated cobordism of Figure~\ref{fig:rp2-tangle-2}
is identically $0$ strictly on the chain level. This corresponds
to the standard $\mathbb{RP}^{2}$ with Euler number $2$; using this,
one can show that $\overline{D}_{\mathsf{BN}}=0$ for any decorated
cobordism $(\Sigma,\vec{A})$ for which $\Sigma$ has a standard Euler
number $2$ $\mathbb{RP}^{2}$-summand.
\end{rem}

\begin{defn}
\label{def:hkh}Consider the full subcategory $\mathsf{HLink}$ of
$\mathsf{Link}_{\{x\}}$ whose objects are links without any basepoints;
then the morphisms are isotopy rel.\ $\partial$ classes of link
cobordisms $\Sigma$ decorated by a first homology class $w\in H_{1}(\Sigma;\mathbb{F})$.
Define 
\[
HKh:=H(DKh\otimes_{\mathbb{F}[x]/(x^{2})}^{L}\mathbb{F}[x]/(x)):\mathsf{HLink}\to\mathsf{Mod}_{\mathbb{F}}
\]
where $H$ means homology and $\mathsf{Mod}_{\mathbb{F}}$ is the
category of bigraded $\mathbb{F}$-modules.
\end{defn}

\begin{lem}
We have $HKh(L)=Kh(L)\otimes_{\mathbb{F}}\Xi_{x}$.
\end{lem}

\begin{proof}
Since $L$ does not have any basepoints, $x\in R_{x}$ acts trivially
on $CKh(L)$. Hence, the differential on $\overline{D}Kh(L,\emptyset)=CKh(L)\otimes_{\mathbb{F}}R_{x}\otimes_{\mathbb{F}}\Xi_{x}$
is $\partial_{CKh(L)}\otimes\mathrm{Id}_{R_{x}\otimes_{\mathbb{F}}\Xi_{x}}+\mathrm{Id}_{CKh(L)}\otimes x\otimes\xi_{x}$,
and so $\overline{D}Kh(L,\emptyset)\otimes_{\mathbb{F}[x]/(x^{2})}\mathbb{F}[x]/(x)=CKh(L)\otimes_{\mathbb{F}}\Xi_{x}$
with differential $\partial_{CKh(L)}\otimes\mathrm{Id}_{\Xi_{x}}$.
\end{proof}
\begin{cor}
\label{cor:Let--be}Let $K$ be a knot in $S^{3}$ and let $\Sigma,\Sigma'$
be two properly embedded orientable surfaces in $D^{4}$ with boundary
$K$, such that $Kh(\Sigma)\neq Kh(\Sigma')$. Then for any $N\ge0$,
we have
\[
HKh(\Sigma\#N\mathbb{RP}^{2},w_{1})=Kh(\Sigma)\otimes\xi_{x}^{N}\neq Kh(\Sigma')\otimes\xi_{x}^{N}=HKh(\Sigma'\#N\mathbb{RP}^{2},w_{1}),
\]
where $\mathbb{RP}^{2}$ denotes the standard $\mathbb{RP}^{2}$ with
Euler number $-2$ and $w_{1}$ is the Poincar\'{e} dual of the first
Stiefel-Whitney class (note that $H_{1}(\Sigma\#N\mathbb{RP}^{2},K;\mathbb{F})\cong H_{1}(\Sigma\#N\mathbb{RP}^{2};\mathbb{F})$).
In particular, $\Sigma\#N\mathbb{RP}^{2}$ and $\Sigma'\#N\mathbb{RP}^{2}$
are not isotopic rel.\ $\partial$.
\end{cor}

\begin{proof}
By Proposition~\ref{prop:composing-tangles-cobordisms} and Lemma~\ref{lem:model-rp2-computation},
we have
\[
\overline{D}Kh(\Sigma\#N\mathbb{RP}^{2},w_{1})=CKh(\Sigma)\otimes\mathrm{Id}_{R_{x}}\otimes\xi_{x}^{N},\ \overline{D}Kh(\Sigma'\#N\mathbb{RP}^{2},w_{1})=CKh(\Sigma')\otimes\mathrm{Id}_{R_{x}}\otimes\xi_{x}^{N}.
\]
Hence, $HKh(\Sigma\#N\mathbb{RP}^{2},w_{1})=Kh(\Sigma)\otimes\xi_{x}^{N}$
and $HKh(\Sigma'\#N\mathbb{RP}^{2},w_{1})=Kh(\Sigma')\otimes\xi_{x}^{N}$.
\end{proof}
\begin{proof}[Proof of Corollary~\ref{cor:examples}]
Note that our convention for Khovanov homology differs from that
of Hayden and Sundberg \cite{MR4726569}; compare \cite[Conventions]{nahm2025khovanov}.
If $\Sigma,\Sigma'$ are the surfaces of \cite[Theorem 1.1]{MR4726569}
in $D^{4}$ with boundary $K\subset S^{3}$, then Hayden and Sundberg
study the \emph{mirrors} $m(\Sigma),m(\Sigma'):K\to\emptyset$ and
the induced link cobordism maps 
\[
Kh(m(\Sigma);\mathbb{Z}),Kh(m(\Sigma');\mathbb{Z}):Kh(K;\mathbb{Z})\to\mathbb{Z}.
\]
They find a homology class $\phi\in Kh(K;\mathbb{Z})$ such that $Kh(m(\Sigma);\mathbb{Z})(\phi)=\pm1$
but $Kh(m(\Sigma');\mathbb{Z})(\phi)=0$ \cite[Section~3]{MR4726569}.
Hence, if $\phi_{\mathbb{F}}\in Kh(K;\mathbb{F})$ is the mod $2$
reduction of $\phi$, then $Kh(m(\Sigma);\mathbb{F})(\phi_{\mathbb{F}})=1$
but $Kh(m(\Sigma');\mathbb{F})(\phi_{\mathbb{F}})=0$. In particular,
$Kh(m(\Sigma);\mathbb{F})\neq Kh(m(\Sigma');\mathbb{F})$, and so
$m(\Sigma)\#N\mathbb{RP}^{2}$ and $m(\Sigma')\#N\mathbb{RP}^{2}$
are not isotopic rel.\ $\partial$ by Corollary~\ref{cor:Let--be}.
\end{proof}

\section{\label{sec:remarks}Two remarks on the enhanced Khovanov TQFT}

\begin{figure}[h]
\begin{centering}
\includegraphics[scale=2.2]{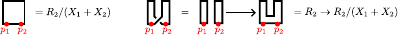}
\par\end{centering}
\caption{\label{fig:two-unknots}The Khovanov chain complexes of these two
link diagrams of the unknot with two basepoints are \emph{not} $R_{2}$-chain
homotopy equivalent.}
\end{figure}

In this paragraph we explain that in order to define the decorated
Khovanov TQFT, it is necessary to work in the derived category $D^{b}(\mathsf{Mod}_{R_{\boldsymbol{X}}})$
instead of the homotopy category $K^{b}(\mathsf{Mod}_{R_{\boldsymbol{X}}})$.
Consider the two generic $2$-pointed links of Figure~\ref{fig:two-unknots}
(the underlying links are the unknot). We would like the Khovanov
TQFT to assign isomorphic objects to them, but the Khovanov chain
complexes of these pointed links are \emph{not} $R_{2}$-homotopy
equivalent (but are $R_{2}$-quasi-isomorphic). Indeed, they are respectively
$R_{2}/(X_{1}+X_{2})$ and the mapping cone of the $R_{2}$-linear
map $R_{2}\rightarrow R_{2}/(X_{1}+X_{2})$ that sends $1$ to $1$;
there is no $R_{2}$-linear chain map from the latter to the former
that is a quasi-isomorphism. The author learned this from Robert Lipshitz;
note that this answers \cite[Remark~2.5]{MR3604486} in the negative.

\begin{figure}[h]
\begin{centering}
\includegraphics[scale=2.2]{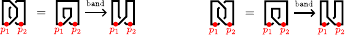}
\par\end{centering}
\caption{\label{fig:u2-h}The Khovanov chain complexes of the unlink with two
components $U_{2}$ and the Hopf link $H$ are mapping cones of a
nonorientable band between two unknots.}
\end{figure}

Next, we explain an observation that led us to Theorems~\ref{thm:main-thm}~and~\ref{thm:enhanced-khovanov}.
Reduced Khovanov homology and knot Floer homology \cite{MR2065507,MR2704683}
are related by spectral sequences \cite{MR4777638,nahm2025spectral}.
However, as discussed in \cite{MR3604486}, it was particularly tricky
to construct such a spectral sequence because, roughly speaking, Khovanov
homology is \emph{unpointed} and link Floer homology is \emph{pointed}.
This difference is more pronounced for links with two or more components,
and one ends up having to reconcile this difference in order to construct
such a spectral sequence (we do not discuss this further in this paper).
Motivated by the need to reconcile this difference, Baldwin, Levine,
and Sarkar defined pointed Khovanov homology, which behaves more similarly
to link Floer homology.

A key example where this distinction between Khovanov and link Floer
homology is visible is the comparison between the two component unlink
$U_{2}$ and the Hopf link $H$. Indeed, the reduced Khovanov homology
of these two links have the same rank, but the link Floer homology
of these two links have different ranks. Pointed Khovanov homology
remedies this difference; if we put one basepoint on each component,
then the ranks of the pointed Khovanov homology of $U_{2}$ and $H$
differ in the same way that the ranks of link Floer homology differ.

The Khovanov chain complexes of $U_{2}$ and $H$ are both the mapping
cone of a standard nonorientable band between two unknots (see Figure~\ref{fig:u2-h});
the only difference is the Euler number of the band. In the unpointed
case, $\mathrm{rk}Kh(U_{2})=\mathrm{rk}Kh(H)=2\mathrm{rk}Kh(U)$ where
$U$ is the unknot (this is another way to see that these nonorientable
bands induce the zero map on unpointed Khovanov homology). On the
other hand, in the pointed Khovanov homology of Baldwin, Levine, and
Sarkar, we have $\mathrm{rk}Kh(U_{2},(p_{1},p_{2}))\neq\mathrm{rk}Kh(H,(p_{1},p_{2}))$.
Hence, if functoriality of pointed Khovanov homology were established,
then the two corresponding \emph{band maps} $Kh(U,(p_{1},p_{2}))\to Kh(U,(p_{1},p_{2}))$
ought to have different ranks. In other words, a functorial theory
for pointed links should distinguish the two standard nonorientable
bands with Euler number $\pm2$ on the unknot. (Functoriality of pointed
Khovanov homology follows from functoriality of $DKh$; see Appendix~\ref{sec:Comparison-with-previous}.)

\appendix

\section{\label{sec:collapsing-colors-proof}A proof of Proposition~\ref{prop:sigma-cheq}}

In this appendix we prove Proposition~\ref{prop:sigma-cheq}. First,
let us prove that the two chain complexes of Remark~\ref{rem:change-basis-1}
are $R_{\boldsymbol{X}}$-isomorphic. We work in a more general setting,
over an additive $\mathbb{F}$-linear category $\mathsf{A}$.
\begin{lem}[Change of basis]
\label{lem:change-basis}Let $(C,\rho_{C})\in K_{R_{\boldsymbol{X}}}^{-}(\mathsf{A})$
be a $R_{\boldsymbol{X}}$-chain complex, and let $(\overline{C},\rho_{\overline{C}})\in K_{R_{\boldsymbol{X}}}^{-}(\mathsf{A})$
be its preferred free $R_{\boldsymbol{X}}$-resolution (Definition~\ref{def:preferred-resolution}).
Recall that $\overline{C}=C\otimes_{\mathbb{F}}R_{\boldsymbol{X}}\otimes_{\mathbb{F}}\Xi_{\boldsymbol{X}}$
and its differential is $\partial_{\overline{C}}:=\partial+\phi$
where 
\[
\partial:=\partial_{C}\otimes\mathrm{Id}_{R_{\boldsymbol{X}}}\otimes\mathrm{Id}_{\Xi_{\boldsymbol{X}}}+\sum_{x\in\boldsymbol{X}}\mathrm{Id}_{C}\otimes x\otimes\xi_{x},\ \phi:=\sum_{x\in\boldsymbol{X}}\rho_{C}(x)\otimes\mathrm{Id}_{R_{\boldsymbol{X}}}\otimes\xi_{x}.
\]
Let $\overline{C}'$ be the $R_{\boldsymbol{X}}$-chain complex with
underlying object $C\otimes_{\mathbb{F}}R_{\boldsymbol{X}}\otimes_{\mathbb{F}}\Xi_{\boldsymbol{X}}$,
differential $\partial_{\overline{C}'}:=\partial$, and $R_{\boldsymbol{X}}$-action
$\rho_{\overline{C}'}$ given by 
\[
\rho_{\overline{C}'}(x):=\rho_{C}(x)\otimes\mathrm{Id}_{R_{\boldsymbol{X}}}\otimes\mathrm{Id}_{\Xi_{\boldsymbol{X}}}+\mathrm{Id}_{C}\otimes x\otimes\mathrm{Id}_{\Xi_{\boldsymbol{X}}}
\]
for $x\in\boldsymbol{X}$.

Define an $\mathbb{F}$-linear map $\psi:C\otimes_{\mathbb{F}}R_{\boldsymbol{X}}\otimes_{\mathbb{F}}\Xi_{\boldsymbol{X}}\to C\otimes_{\mathbb{F}}R_{\boldsymbol{X}}\otimes_{\mathbb{F}}\Xi_{\boldsymbol{X}}$
as follows. For $\boldsymbol{Z}\subset\boldsymbol{X}$, let $x_{\boldsymbol{Z}}:=\prod_{z\in\boldsymbol{Z}}x_{z}$.
As an $\mathbb{F}$-vector space, $R_{\boldsymbol{X}}$ has basis
$\{x_{\boldsymbol{Z}}\}_{\boldsymbol{Z}\subset\boldsymbol{X}}$; hence
$C\otimes_{\mathbb{F}}R_{\boldsymbol{X}}\otimes_{\mathbb{F}}\Xi_{\boldsymbol{X}}$
splits as $\bigoplus_{\boldsymbol{Z}\subset\boldsymbol{X}}C\otimes_{\mathbb{F}}x_{\boldsymbol{Z}}\mathbb{F}\otimes_{\mathbb{F}}\Xi_{\boldsymbol{X}}$.
The $C\otimes_{\mathbb{F}}x_{\boldsymbol{Z}}\mathbb{F}\otimes_{\mathbb{F}}\Xi_{\boldsymbol{X}}\to C\otimes_{\mathbb{F}}R_{\boldsymbol{X}}\otimes_{\mathbb{F}}\Xi_{\boldsymbol{X}}$
component of $\psi$ is the composition 
\[
C\otimes_{\mathbb{F}}x_{\boldsymbol{Z}}\mathbb{F}\otimes_{\mathbb{F}}\Xi_{\boldsymbol{X}}\xrightarrow{\mathrm{Id}_{C}\otimes x_{\boldsymbol{Z}}^{-1}\otimes\mathrm{Id}_{\Xi_{\boldsymbol{X}}}}C\otimes_{\mathbb{F}}1\mathbb{F}\otimes_{\mathbb{F}}\Xi_{\boldsymbol{X}}\xrightarrow{\rho_{\overline{C}'}(x_{\boldsymbol{Z}})}C\otimes_{\mathbb{F}}R_{\boldsymbol{X}}\otimes_{\mathbb{F}}\Xi_{\boldsymbol{X}},
\]
where the first map is given by the $\mathbb{F}$-linear isomorphism
between $x_{\boldsymbol{Z}}\mathbb{F}$ and $1\mathbb{F}$ which we
denote as $x_{\boldsymbol{Z}}^{-1}$.

Then, $\psi$ induces $R_{\boldsymbol{X}}$-chain isomorphisms $(\overline{C},\rho_{\overline{C}})\to(\overline{C}',\rho_{\overline{C}'})$
and $(\overline{C}',\rho_{\overline{C}'})\to(\overline{C},\rho_{\overline{C}})$.
In fact, it is its own inverse.
\end{lem}

\begin{proof}
First, let us check that $\overline{C}'$ is an $R_{\boldsymbol{X}}$-chain
complex. It is straightforward to check $\partial^{2}=0$. To check
$R_{\boldsymbol{X}}$-linearity, it is sufficient to check $\partial\rho_{\overline{C}'}(x)=\rho_{\overline{C'}}(x)\partial$
for all $x\in\boldsymbol{X}$, which follows from $\partial_{C}\rho_{C}(x)=\rho_{C}(x)\partial_{C}.$%

Now, let us check that $\psi^{2}=\mathrm{Id}_{C}\otimes\mathrm{Id}_{R_{\boldsymbol{X}}}\otimes\mathrm{Id}_{\Xi_{\boldsymbol{X}}}$.
For this, let us study the $C\otimes_{\mathbb{F}}x_{\boldsymbol{Z}}\mathbb{F}\otimes_{\mathbb{F}}\Xi_{\boldsymbol{X}}\to C\otimes_{\mathbb{F}}x_{\boldsymbol{Y}}\mathbb{F}\otimes_{\mathbb{F}}\Xi_{\boldsymbol{X}}$
component of $\psi$. Under the identification $x_{\boldsymbol{Z}}\mathbb{F}\cong1\mathbb{F}\cong x_{\boldsymbol{Y}}\mathbb{F}$
as $\mathbb{F}$-vector spaces, we view this as a map $C\otimes_{\mathbb{F}}1\mathbb{F}\otimes_{\mathbb{F}}\Xi_{\boldsymbol{X}}\to C\otimes_{\mathbb{F}}1\mathbb{F}\otimes_{\mathbb{F}}\Xi_{\boldsymbol{X}}$.
This map is $0$ if $\boldsymbol{Y}\not\subset\boldsymbol{Z}$, and
is $\rho_{C}(x_{\boldsymbol{Z}\setminus\boldsymbol{Y}})\otimes\mathrm{Id}_{1\mathbb{F}}\otimes\mathrm{Id}_{\Xi_{\boldsymbol{X}}}$
otherwise. Hence, the $C\otimes_{\mathbb{F}}x_{\boldsymbol{Z}}\mathbb{F}\otimes_{\mathbb{F}}\Xi_{\boldsymbol{X}}\to C\otimes_{\mathbb{F}}x_{\boldsymbol{Y}}\mathbb{F}\otimes_{\mathbb{F}}\Xi_{\boldsymbol{X}}$
component of $\psi^{2}$ is $0$ if $\boldsymbol{Y}\not\subset\boldsymbol{Z}$,
and is otherwise (similarly viewed as $C\otimes_{\mathbb{F}}1\mathbb{F}\otimes_{\mathbb{F}}\Xi_{\boldsymbol{X}}\to C\otimes_{\mathbb{F}}1\mathbb{F}\otimes_{\mathbb{F}}\Xi_{\boldsymbol{X}}$)
\[
\sum_{\boldsymbol{Y}\subset\boldsymbol{Y}'\subset\boldsymbol{Z}}\rho_{C}(x_{\boldsymbol{Y}'\setminus\boldsymbol{Y}})\rho_{C}(x_{\boldsymbol{Z}\setminus\boldsymbol{Y}'})\otimes\mathrm{Id}_{1\mathbb{F}}\otimes\mathrm{Id}_{\Xi_{\boldsymbol{X}}}=2^{|\boldsymbol{Z}\setminus\boldsymbol{Y}|}\rho_{C}(x_{\boldsymbol{Z}\setminus\boldsymbol{Y}})\otimes\mathrm{Id}_{1\mathbb{F}}\otimes\mathrm{Id}_{\Xi_{\boldsymbol{X}}}=\begin{cases}
\mathrm{Id} & \mathrm{if\ }\boldsymbol{Y}=\boldsymbol{Z}\\
0 & {\rm otherwise}
\end{cases}.
\]

Finally, let us check that $\psi:(\overline{C},\rho_{\overline{C}})\to(\overline{C}',\rho_{\overline{C}'})$
is (1) $R_{\boldsymbol{X}}$-linear, and (2) a chain map. From this
and the previous paragraph, that $\psi:(\overline{C}',\rho_{\overline{C}'})\to(\overline{C},\rho_{\overline{C}})$
is an $R_{\boldsymbol{X}}$-linear chain map and that they are isomorphisms
follow automatically.

For (1), it is sufficient to check $\psi\rho_{\overline{C}}(x)=\rho_{\overline{C}'}(x)\psi$
for all $x\in\boldsymbol{X}$. Let us check this on each summand $C\otimes_{\mathbb{F}}x_{\boldsymbol{Z}}\mathbb{F}\otimes_{\mathbb{F}}\Xi_{\boldsymbol{X}}$
separately. If $xx_{\boldsymbol{Z}}\neq0$, then 
\begin{align*}
\rho_{\overline{C}'}(x)\psi & =\rho_{\overline{C}'}(x)\rho_{\overline{C}'}(x_{\boldsymbol{Z}})(\mathrm{Id}_{C}\otimes x_{\boldsymbol{Z}}^{-1}\otimes\mathrm{Id}_{\Xi_{\boldsymbol{X}}})\\
 & =\rho_{\overline{C}'}(xx_{\boldsymbol{Z}})(\mathrm{Id}_{C}\otimes(xx_{\boldsymbol{Z}})^{-1}\otimes\mathrm{Id}_{\Xi_{\boldsymbol{X}}})(\mathrm{Id}_{C}\otimes x\otimes\mathrm{Id}_{\Xi_{\boldsymbol{X}}})=\psi\rho_{\overline{C}}(x).
\end{align*}
If $xx_{\boldsymbol{Z}}=0$, then they both vanish.

For (2), we would like to check $\partial_{\overline{C}'}\psi=\psi\partial_{\overline{C}}$.
First, it is straightforward to check that they agree on $C\otimes_{\mathbb{F}}1\mathbb{F}\otimes_{\mathbb{F}}\Xi_{\boldsymbol{X}}$.
Given this, they agree on $C\otimes_{\mathbb{F}}x_{\boldsymbol{Z}}\mathbb{F}\otimes_{\mathbb{F}}\Xi_{\boldsymbol{X}}$
since on $C\otimes_{\mathbb{F}}1\mathbb{F}\otimes_{\mathbb{F}}\Xi_{\boldsymbol{X}}$
we have (the first and fourth equalities follow from the previous
paragraph)
\[
\partial_{\overline{C}'}\psi\rho_{\overline{C}}(x_{\boldsymbol{Z}})=\partial_{\overline{C}'}\rho_{\overline{C}'}(x_{\boldsymbol{Z}})\psi=\rho_{\overline{C}'}(x_{\boldsymbol{Z}})\partial_{\overline{C}'}\psi=\rho_{\overline{C}'}(x_{\boldsymbol{Z}})\psi\partial_{\overline{C}}=\psi\rho_{\overline{C}}(x_{\boldsymbol{Z}})\partial_{\overline{C}}=\psi\partial_{\overline{C}}\rho_{\overline{C}}(x_{\boldsymbol{Z}})
\]
and $\rho_{\overline{C}}(x_{\boldsymbol{Z}})$ maps $C\otimes_{\mathbb{F}}1\mathbb{F}\otimes_{\mathbb{F}}\Xi_{\boldsymbol{X}}$
isomorphically to $C\otimes_{\mathbb{F}}x_{\boldsymbol{Z}}\mathbb{F}\otimes_{\mathbb{F}}\Xi_{\boldsymbol{X}}$.
\end{proof}
Let us use the notations of Definition~\ref{def:sigma-color}. First
we record the following equations.
\begin{gather}
\xi_{w}\circ\Xi_{\sigma}=\Xi_{\sigma}\circ\xi_{\sigma(w)},\label{eq:xi-identity}\\
R_{\sigma}\circ x=R_{\sigma}(x)\circ R_{\sigma}=\sum_{\sigma(w)=x}w\circ R_{\sigma}.\label{eq:r-identity}
\end{gather}

Let us first show that $\overline{\sigma}:\overline{C}_{\boldsymbol{X}}\to\overline{C}_{\boldsymbol{W}}$
is an $R_{\boldsymbol{X}}$-linear chain map. That it is $R_{\boldsymbol{X}}$-linear
follows from Equation~(\ref{eq:r-identity}). Let us check that it
is a chain map.
\begin{align}
\partial_{\overline{C}_{\boldsymbol{W}}}\overline{\sigma} & =\partial_{C}\otimes R_{\sigma}\otimes\Xi_{\sigma}+\sum_{w\in\boldsymbol{W}}(\rho_{C}(w)\otimes R_{\sigma}+\mathrm{Id}_{C}\otimes(wR_{\sigma}))\otimes(\xi_{w}\Xi_{\sigma})\label{eq:sigma-chain}\\
 & =\partial_{C}\otimes R_{\sigma}\otimes\Xi_{\sigma}+\sum_{x\in\boldsymbol{X}}\left(\sum_{\sigma(w)=x}(\rho_{C}(w)\otimes R_{\sigma}+\mathrm{Id}_{C}\otimes(wR_{\sigma}))\right)\otimes(\Xi_{\sigma}\xi_{x})\nonumber \\
 & =\partial_{C}\otimes R_{\sigma}\otimes\Xi_{\sigma}+\sum_{x\in\boldsymbol{X}}(\rho_{C}(R_{\sigma}(x))\otimes R_{\sigma}+\mathrm{Id}_{C}\otimes(R_{\sigma}x))\otimes(\Xi_{\sigma}\xi_{x})=\overline{\sigma}\partial_{\overline{C}_{\boldsymbol{X}}}.\nonumber 
\end{align}
Here the second equality follows from Equation~(\ref{eq:xi-identity}).

Note that $R_{\sigma}\otimes\Xi_{\sigma}:F_{\boldsymbol{X}}\to F_{\boldsymbol{W}}$
is an $R_{\boldsymbol{X}}$-chain homotopy equivalence, since $R_{\sigma}\otimes\Xi_{\sigma}$
is a quasi-isomorphism and $F_{\boldsymbol{X}},F_{\boldsymbol{W}}$
are $R_{\boldsymbol{X}}$-free. Let $G:F_{\boldsymbol{W}}\to F_{\boldsymbol{X}}$
be an $R_{\boldsymbol{X}}$-homotopy inverse; one may hope that $\mathrm{Id}_{C}\otimes G:\overline{C}_{\boldsymbol{W}}\to\overline{C}_{\boldsymbol{X}}$
is an $R_{\boldsymbol{X}}$-homotopy inverse of $\overline{\sigma}$.
However, $\mathrm{Id}_{C}\otimes G$ is a priori not even a chain
map. The trick is to use Lemma~\ref{lem:change-basis} and perform
a change of basis.

Denote as $(\overline{C}_{\boldsymbol{W}}',\rho_{\overline{C}_{\boldsymbol{W}}'})$
(resp.\ $(\overline{C}_{\boldsymbol{X}}',\rho_{\overline{C}_{\boldsymbol{X}}'})$)
the chain complex defined as in Lemma~\ref{lem:change-basis} for
$\overline{C}_{\boldsymbol{W}}$ (resp.\ $\overline{C}_{\boldsymbol{X}}$),
and denote as $\psi_{\boldsymbol{W}}:C\otimes_{\mathbb{F}}R_{\boldsymbol{W}}\otimes_{\mathbb{F}}\Xi_{\boldsymbol{W}}\to C\otimes_{\mathbb{F}}R_{\boldsymbol{W}}\otimes_{\mathbb{F}}\Xi_{\boldsymbol{W}}$
(resp.\ $\psi_{\boldsymbol{X}}:C\otimes_{\mathbb{F}}R_{\boldsymbol{X}}\otimes_{\mathbb{F}}\Xi_{\boldsymbol{X}}\to C\otimes_{\mathbb{F}}R_{\boldsymbol{X}}\otimes_{\mathbb{F}}\Xi_{\boldsymbol{X}}$)
the corresponding isomorphism.

Interestingly, $\overline{\sigma}:=\mathrm{Id}_{C}\otimes R_{\sigma}\otimes\Xi_{\sigma}:\overline{C}_{\boldsymbol{X}}'\to\overline{C}_{\boldsymbol{W}}'$
is also an $R_{\boldsymbol{X}}$-linear chain map. To check that it
is $R_{\boldsymbol{X}}$-linear, it is sufficient to check $\overline{\sigma}\rho_{\overline{C}_{\boldsymbol{X}}'}(x)=\rho_{\overline{C}_{\boldsymbol{W}}'}(R_{\sigma}(x))\overline{\sigma}$
for all $x\in\boldsymbol{X}$. We have
\begin{align*}
\overline{\sigma}\rho_{\overline{C}_{\boldsymbol{X}}'}(x) & =\rho_{C}(R_{\sigma}(x))\otimes R_{\sigma}\otimes\Xi_{\sigma}+\mathrm{Id}_{C}\otimes R_{\sigma}x\otimes\Xi_{\sigma}\\
 & =\rho_{C}(R_{\sigma}(x))\otimes R_{\sigma}\otimes\Xi_{\sigma}+\mathrm{Id}_{C}\otimes R_{\sigma}(x)R_{\sigma}\otimes\Xi_{\sigma}=\rho_{\overline{C}_{\boldsymbol{W}}'}(R_{\sigma}(x))\overline{\sigma}
\end{align*}
where the second equality follows from Equation~(\ref{eq:r-identity}).
That it is a chain map is similar to Equation~(\ref{eq:sigma-chain})
but simpler:
\begin{align*}
\partial_{\overline{C}_{\boldsymbol{W}}'}\overline{\sigma} & =\partial_{C}\otimes R_{\sigma}\otimes\Xi_{\sigma}+\sum_{w\in\boldsymbol{W}}\mathrm{Id}_{C}\otimes(wR_{\sigma})\otimes(\xi_{w}\Xi_{\sigma})\\
 & =\partial_{C}\otimes R_{\sigma}\otimes\Xi_{\sigma}+\sum_{x\in\boldsymbol{X}}\left(\sum_{\sigma(w)=x}\mathrm{Id}_{C}\otimes(wR_{\sigma})\right)\otimes(\Xi_{\sigma}\xi_{x})\\
 & =\partial_{C}\otimes R_{\sigma}\otimes\Xi_{\sigma}+\sum_{x\in\boldsymbol{X}}\mathrm{Id}_{C}\otimes(R_{\sigma}x)\otimes(\Xi_{\sigma}\xi_{x})=\overline{\sigma}\partial_{\overline{C}_{\boldsymbol{X}}'}.
\end{align*}

We claim that $\psi_{\boldsymbol{W}}\overline{\sigma}\psi_{\boldsymbol{X}}=\overline{\sigma}$.
It is easy to show that they agree on $C\otimes_{\mathbb{F}}1\mathbb{F}\otimes_{\mathbb{F}}\Xi_{\boldsymbol{X}}$.
Given this, on $C\otimes_{\mathbb{F}}1\mathbb{F}\otimes_{\mathbb{F}}\Xi_{\boldsymbol{X}}$
we have for any $\boldsymbol{Z}\subset\boldsymbol{X}$
\begin{multline*}
\psi_{\boldsymbol{W}}\overline{\sigma}\psi_{\boldsymbol{X}}\rho_{\overline{C}_{\boldsymbol{X}}'}(x_{\boldsymbol{Z}})=\psi_{\boldsymbol{W}}\overline{\sigma}\rho_{\overline{C}_{\boldsymbol{X}}}(x_{\boldsymbol{Z}})\psi_{\boldsymbol{X}}=\psi_{\boldsymbol{W}}\rho_{\overline{C}_{\boldsymbol{W}}}(x_{\boldsymbol{Z}})\overline{\sigma}\psi_{\boldsymbol{X}}\\
=\rho_{\overline{C}_{\boldsymbol{W}}'}(x_{\boldsymbol{Z}})\psi_{\boldsymbol{W}}\overline{\sigma}\psi_{\boldsymbol{X}}=\rho_{\overline{C}_{\boldsymbol{W}}'}(x_{\boldsymbol{Z}})\overline{\sigma}=\overline{\sigma}\rho_{\overline{C}_{\boldsymbol{X}}'}(x_{\boldsymbol{Z}}),
\end{multline*}
and so $\psi_{\boldsymbol{W}}\overline{\sigma}\psi_{\boldsymbol{X}}=\overline{\sigma}$
on $C\otimes_{\mathbb{F}}x_{\boldsymbol{Z}}\mathbb{F}\otimes_{\mathbb{F}}\Xi_{\boldsymbol{X}}$.

Let us record the following simple lemma.
\begin{lem}
\label{lem:simple-appendix}Let $D_{1},D_{2}\in Ch^{-}(\mathsf{Mod}_{R_{\boldsymbol{X}}})$
be chain complexes of $R_{\boldsymbol{X}}$-modules, and let $f:D_{1}\to D_{2}$
be an $R_{\boldsymbol{X}}$-linear map. Let $(C,\rho_{C})\in K_{R_{\boldsymbol{X}}}^{-}(\mathsf{A})$
be an $R_{\boldsymbol{X}}$-chain complex over $\mathsf{A}$. Equip
$C\otimes_{\mathbb{F}}D_{i}$ with the $R_{\boldsymbol{X}}$-action
$\rho_{i}$ given by $\rho_{i}(x):=\rho_{C}(x)\otimes\mathrm{Id}_{D_{i}}+\mathrm{Id}_{C}\otimes x$
for $x\in\boldsymbol{X}$. Then, the map $\mathrm{Id}_{C}\otimes f:C\otimes_{\mathbb{F}}D_{1}\to C\otimes_{\mathbb{F}}D_{2}$
is $R_{\boldsymbol{X}}$-linear.
\end{lem}

\begin{proof}
We check $(\mathrm{Id}_{C}\otimes f)\rho_{1}(x)=\rho_{2}(x)(\mathrm{Id}_{C}\otimes f)$
for all $x\in\boldsymbol{X}$, which is immediate.
\end{proof}
Now, let $F:=R_{\sigma}\otimes\Xi_{\sigma}:F_{\boldsymbol{X}}\to F_{\boldsymbol{W}}$,
let $G:F_{\boldsymbol{W}}\to F_{\boldsymbol{X}}$ be an $R_{\boldsymbol{X}}$-homotopy
inverse of $F$, and let $H,H'$ be $R_{\boldsymbol{X}}$-linear maps
such that $[\partial_{F_{\boldsymbol{X}}},H]=GF+\mathrm{Id}$ and
$[\partial_{F_{\boldsymbol{W}}},H']=FG+\mathrm{Id}$. Then, consider
the maps $\overline{\sigma}=\mathrm{Id}_{C}\otimes F:\overline{C}_{\boldsymbol{X}}'\to\overline{C}_{\boldsymbol{W}}'$,
$\mathrm{Id}_{C}\otimes G:\overline{C}_{\boldsymbol{W}}'\to\overline{C}_{\boldsymbol{X}}'$,
$\mathrm{Id}_{C}\otimes H:\overline{C}_{\boldsymbol{X}}'\to\overline{C}_{\boldsymbol{X}}'$,
and $\mathrm{Id}_{C}\otimes H':\overline{C}_{\boldsymbol{W}}'\to\overline{C}_{\boldsymbol{W}}'$;
they are all $R_{\boldsymbol{X}}$-linear by Lemma~\ref{lem:simple-appendix}.
Furthermore, $[\partial_{\overline{C}_{\boldsymbol{X}}'},\mathrm{Id}_{C}\otimes H]=(\mathrm{Id}_{C}\otimes G)(\mathrm{Id}_{C}\otimes F)+\mathrm{Id}_{\overline{C}_{\boldsymbol{X}}'}$,
and $[\partial_{\overline{C}_{\boldsymbol{W}}'},\mathrm{Id}_{C}\otimes H']=(\mathrm{Id}_{C}\otimes F)(\mathrm{Id}_{C}\otimes G)+\mathrm{Id}_{\overline{C}_{\boldsymbol{W}}'}$.

Hence, $\psi_{\boldsymbol{X}}(\mathrm{Id}_{C}\otimes G)\psi_{\boldsymbol{W}}$
is an $R_{\boldsymbol{X}}$-homotopy inverse of $\overline{\sigma}$,
and $\psi_{\boldsymbol{X}}(\mathrm{Id}_{C}\otimes H)\psi_{\boldsymbol{X}}$
and $\psi_{\boldsymbol{W}}(\mathrm{Id}_{C}\otimes H')\psi_{\boldsymbol{W}}$
are the homotopies. This completes the proof of Proposition~\ref{prop:sigma-cheq}.

\section{\label{sec:Comparison-with-previous}Comparison with previous works}

In this appendix we do some bookkeeping to show that our Khovanov
decorated link cobordism maps extend Baldwin-Levine-Sarkar and Lipshitz-Sarkar's
theories \cite{MR3604486,MR4521052} into TQFTs. In particular, we
prove Theorem~\ref{thm:enhanced-khovanov}~(\ref{enu:For-the-decorated}).

\subsection{Pointed Khovanov homology}

Let us recall the definition of pointed Khovanov homology of Baldwin,
Levine, and Sarkar \cite{MR3604486} for generic $\boldsymbol{X}$-pointed
links $(L,\vec{\boldsymbol{p}})$ (they consider the case where all
the $\boldsymbol{p}_{x}$'s are singletons, but this will not be important).
They first define a chain complex $K_{\boldsymbol{X}}:=\bigotimes_{x\in\boldsymbol{X}}(R_{\boldsymbol{X}}\xrightarrow{x}\xi_{x}R_{\boldsymbol{X}})$
where the tensor product is taken over $R_{\boldsymbol{X}}$ and $\xi_{x}$
is a formal variable of bidegree $(1,2)$. Let $\Lambda_{\boldsymbol{X}}:=\Lambda^{\ast}(\{\xi_{x}\}_{x\in\boldsymbol{X}})$
be the exterior algebra\footnote{Since we are in characteristic $2$, we have $\Lambda_{\boldsymbol{X}}=R_{\boldsymbol{X}}$.}
on generators $\{\xi_{x}\}_{x\in\boldsymbol{X}}$; then $K_{\boldsymbol{X}}$
is a $\Lambda_{\boldsymbol{X}}$-chain complex. The\emph{ pointed
Khovanov chain complex} is the $\Lambda_{\boldsymbol{X}}$-chain complex
$DKh(L,\vec{\boldsymbol{p}})\otimes_{R_{\boldsymbol{X}}}K_{\boldsymbol{X}}$.

Let us explain how to define decorated link cobordism maps for $DKh(L,\vec{\boldsymbol{p}})\otimes_{R_{\boldsymbol{X}}}K_{\boldsymbol{X}}$.
If $(\Sigma,\vec{A}):(L,\vec{\boldsymbol{p}})\to(L',\vec{\boldsymbol{p}}')$
is an $\boldsymbol{X}$-decorated link cobordism, then our $R_{\boldsymbol{X}}$-linear
Khovanov link cobordism chain map $\overline{D}Kh(\Sigma,\vec{A}):\overline{D}Kh(L,\vec{\boldsymbol{p}})\to\overline{D}Kh(L',\vec{\boldsymbol{p}}')$
induces a $\Lambda_{\boldsymbol{X}}$-linear chain map 
\begin{equation}
\overline{D}Kh(\Sigma,\vec{A})\otimes\mathrm{Id}_{K_{\boldsymbol{X}}}:\overline{D}Kh(L,\vec{\boldsymbol{p}})\otimes_{R_{\boldsymbol{X}}}K_{\boldsymbol{X}}\to\overline{D}Kh(L',\vec{\boldsymbol{p}}')\otimes_{R_{\boldsymbol{X}}}K_{\boldsymbol{X}}.\label{eq:dkhkx}
\end{equation}
Recall from Definition~\ref{def:preferred} the $R_{\boldsymbol{X}}$-quasi-isomorphism
$q_{DKh(L,\vec{\boldsymbol{p}})}:\overline{D}Kh(L,\vec{\boldsymbol{p}})\to DKh(L,\vec{\boldsymbol{p}})$;
this induces an $\Lambda_{\boldsymbol{X}}$-quasi-isomorphism 
\begin{equation}
q_{DKh(L,\vec{\boldsymbol{p}})}\otimes\mathrm{Id}_{K_{\boldsymbol{X}}}:\overline{D}Kh(L,\vec{\boldsymbol{p}})\otimes_{R_{\boldsymbol{X}}}K_{\boldsymbol{X}}\to DKh(L,\vec{\boldsymbol{p}})\otimes_{R_{\boldsymbol{X}}}K_{\boldsymbol{X}}.\label{eq:qiso-kx}
\end{equation}

Since these chain complexes are $\Lambda_{\boldsymbol{X}}$-free,
Equation~(\ref{eq:qiso-kx}) is a $\Lambda_{\boldsymbol{X}}$-chain
homotopy equivalence. Let us define (although this is unnecessary)
a $\Lambda_{\boldsymbol{X}}$-homotopy inverse $f:DKh(L,\vec{\boldsymbol{p}})\otimes_{R_{\boldsymbol{X}}}K_{\boldsymbol{X}}\to\overline{D}Kh(L,\vec{\boldsymbol{p}})\otimes_{R_{\boldsymbol{X}}}K_{\boldsymbol{X}}$
of Equation~(\ref{eq:qiso-kx}) as follows. For $\boldsymbol{Z}\subset\boldsymbol{X}$,
let $\xi_{\boldsymbol{Z}}:=\prod_{z\in\boldsymbol{Z}}\xi_{z}$. For
$d\in DKh(L,\vec{\boldsymbol{p}})$, define
\[
f(d\otimes\xi_{\boldsymbol{Z}}):=\sum_{\boldsymbol{Y}\subset\boldsymbol{X}\setminus\boldsymbol{Z}}(d\otimes1\otimes\xi_{\boldsymbol{Y}}^{-1})\otimes\xi_{\boldsymbol{Y}\sqcup\boldsymbol{Z}}.
\]
One can check that this is a $\Lambda_{\boldsymbol{X}}$-chain map
and that $(q_{DKh(L,\vec{\boldsymbol{p}})}\otimes\mathrm{Id}_{K_{\boldsymbol{X}}})\circ f=\mathrm{Id}_{DKh(L,\vec{\boldsymbol{p}})\otimes_{R_{\boldsymbol{X}}}K_{\boldsymbol{X}}}$.

Define the decorated link cobordism map for $(\Sigma,\vec{A})$ on
the pointed Khovanov chain complex as the $\Lambda_{\boldsymbol{X}}$-chain
map 
\begin{equation}
(q_{DKh(L,\vec{\boldsymbol{p}})}\otimes\mathrm{Id}_{K_{\boldsymbol{X}}})\circ(\overline{D}Kh(\Sigma,\vec{A})\otimes\mathrm{Id}_{K_{\boldsymbol{X}}})\circ f:DKh(L,\vec{\boldsymbol{p}})\otimes_{R_{\boldsymbol{X}}}K_{\boldsymbol{X}}\to DKh(L',\vec{\boldsymbol{p}}')\otimes_{R_{\boldsymbol{X}}}K_{\boldsymbol{X}}.\label{eq:pointed-map}
\end{equation}

Baldwin, Levine, and Sarkar show \cite[Proposition~2.9]{MR3604486}
that if $(L',\vec{\boldsymbol{p}'})$ is a generic $\boldsymbol{X}$-pointed
link that is isotopic to $(L,\vec{\boldsymbol{p}})$, then $DKh(L,\vec{\boldsymbol{p}})\otimes_{R_{\boldsymbol{X}}}K_{\boldsymbol{X}}$
and $DKh(L',\vec{\boldsymbol{p}}')\otimes_{R_{\boldsymbol{X}}}K_{\boldsymbol{X}}$
are $\Lambda_{\boldsymbol{X}}$-chain homotopy equivalent by constructing
maps for Reidemeister and slide movies. It is straightforward to check
that their maps agree with Equation~(\ref{eq:pointed-map}) for Reidemeister
and slide movies $(\Sigma,\vec{A})$.

\subsection{Khovanov homology as an $A_{\infty}$-module}

Lipshitz and Sarkar \cite[Theorem~4.2]{MR4521052} consider generic
$\boldsymbol{X}$-pointed links $(L,\vec{\boldsymbol{p}})$ for $|\boldsymbol{X}|=2$;
let $\boldsymbol{X}=\{w,x\}$. (Also, $\boldsymbol{p}_{w}$ and $\boldsymbol{p}_{x}$
are singletons, but this will not be important.) They show that if
$(L',\vec{\boldsymbol{p}'})$ is also such a generic $\boldsymbol{X}$-pointed
link that is isotopic to $(L,\vec{\boldsymbol{p}})$, then $DKh(L,\vec{\boldsymbol{p}})$
and $DKh(L',\vec{\boldsymbol{p}'})$ are $R_{\boldsymbol{X}}$-quasi-isomorphic,
by considering them as $A_{\infty}$-bimodules over $\mathbb{F}[w]/(w^{2})$
and $\mathbb{F}[x]/(x^{2})$ (with trivial higher actions), and defining
an $A_{\infty}$-homotopy equivalence between them.

We are left to compare our map with theirs for slide movies (that
the Reidemeister movies give the same maps are clear). For this, we
specify our preferred $A_{\infty}$-homotopy inverse $f:DKh(L,\vec{\boldsymbol{p}})\to\overline{D}Kh(L,\vec{\boldsymbol{p}})$
of our preferred $R_{\boldsymbol{X}}$-quasi-isomorphism $q_{DKh(L,\vec{\boldsymbol{p}})}:\overline{D}Kh(L,\vec{\boldsymbol{p}})\to DKh(L,\vec{\boldsymbol{p}})$
from Definition~\ref{def:preferred}. Let
\[
f_{m,1,n}:(\mathbb{F}[w]/(w^{2}))^{\otimes m}\otimes DKh(L,\vec{\boldsymbol{p}})\otimes(\mathbb{F}[x]/(x^{2}))^{\otimes n}\to\overline{D}Kh(L,\vec{\boldsymbol{p}})
\]
be such that for $w_{i}\in\{1,w\}$, $x_{j}\in\{1,x\}$, and $d\in DKh(L,\vec{\boldsymbol{p}})$,
we have
\[
f_{m,1,n}(w_{1}\otimes\cdots\otimes w_{m}\otimes d\otimes x_{1}\otimes\cdots\otimes x_{n})=\begin{cases}
d\otimes1\otimes\xi_{w}^{-m}\xi_{x}^{-n} & \mathrm{if}\ \forall i\ w_{i}=w,\ \forall j\ x_{j}=x\\
0 & \mathrm{otherwise}
\end{cases}.
\]

Now, we can check directly that if $(\Sigma,\vec{A}):(L,\vec{\boldsymbol{p}})\to(L',\vec{\boldsymbol{p}}')$
is a slide movie, then $q_{DKh(L,\vec{\boldsymbol{p}}')}\circ\overline{D}Kh(\Sigma,\vec{A})\circ f$
agrees with Lipshitz and Sarkar's $A_{\infty}$-homotopy equivalence
$DKh(L,\vec{\boldsymbol{p}})\to DKh(L',\vec{\boldsymbol{p}'})$.

\bibliographystyle{amsalpha}
\bibliography{/Users/gheehyun/Documents/writings/bib}

\end{document}